\RequirePackage{fix-cm}
\documentclass[smallextended]{svjour3}     
\smartqed  
\usepackage{amsmath,amssymb}
\usepackage[colorlinks,linkcolor=blue,anchorcolor=blue,citecolor=blue]{hyperref}
\usepackage{graphicx}
\usepackage{graphics}
\usepackage{subfigure}
\usepackage{pdflscape,epstopdf}
\usepackage{float}
\usepackage{cases}
\usepackage[ruled]{algorithm2e}
\usepackage{multirow}
\usepackage{array}
\usepackage{rotating}
\usepackage{enumitem}
\usepackage{enumitem}
\usepackage{bm}

\def\cX{{\mathcal X}}
\def\cY{{\mathcal Y}}
\def\cZ{{\mathcal Z}}
\def\cS{{\mathcal S}}
\def\cT{{\mathcal T}}

\begin{document}

\title{ 
An accelerated semi-proximal ADMM with applications to multi-block sparse optimization problems\thanks{This work was funded by the National Key R \& D Program of China (No. 2021YFA001300), the
National Natural Science Foundation of China (No. 12271150), the Hunan Provincial Natural Science
Foundation of China (No. 2023JJ10001), the Science and Technology Innovation Program of Hunan
Province (No. 2022RC1190), and the Hunan Provincial Innovation Foundation for
Postgraduate (No. CX20220432).}}

\titlerunning{Accelerated Semi-Proximal ADMM}
\author{Peng Liu\and Liang Chen\and Minru Bai} 
\authorrunning{Peng Liu et al.} 

\institute{
Peng Liu\at  
School of Mathematics, Hunan University, Changsha 410082, China\\
\email{liupen@hnu.edu.cn} \and
Liang Chen\at
School of Mathematics, Hunan University, Changsha 410082, China\\
 \email{chl@hnu.edu.cn} \and
Minru Bai, Corresponding Author\at 
School of Mathematics, Hunan University, Changsha 410082, China\\
\email{minru-bai@hnu.edu.cn} 
}

\graphicspath{ {Images/} }
\date{Received: date / Accepted: date}
 
\maketitle

\begin{abstract}
As an extension of the alternating direction method of multipliers (ADMM), the semi-proximal ADMM (sPADMM) has been widely used in various fields due to its flexibility and robustness.  
In this paper, we first show that the two-block sPADMM algorithm can achieve an $O(1/\sqrt{K})$ non-ergodic convergence rate.
Then we propose an accelerated sPADMM (AsPADMM) algorithm by introducing extrapolation techniques and incrementing penalty parameters. 
The proposed AsPADMM algorithm is proven to converge globally to an optimal solution with a non-ergodic convergence rate of $O(1/K)$.
Furthermore, the AsPADMM can be extended and combined with the symmetric Gauss-Seidel decomposition to achieve an accelerated ADMM for multi-block problems.  
Finally, we apply the proposed AsPADMM to solving the multi-block subproblems in difference-of-convex algorithms for robust low-rank tensor completion problems and mixed sparse optimization problems.
The numerical results suggest that the acceleration techniques bring about a notable improvement in the convergence speed.

\end{abstract}

\section{Introduction}
Consider the two-block convex optimization problem with equality constraints in the following form
\begin{equation}
\label{problem 1}
\min_{x\in\cX,y\in \cY} \{f(x)+g(y)\mid Ax+By=c\},
\end{equation}
where $\cX$, $\cY$ and $\cZ$ are finite-dimensional real Hilbert space each endowed with an inner product $\langle \cdot,\cdot\rangle$ and its induced norm $\|\cdot\|$, 
$f:\cX\rightarrow (-\infty,\infty ]$ and 
$g:\cY\rightarrow (-\infty,\infty ]$ are closed proper convex functions,
$A: \cX \to\cZ$ and $B:\cY\to\cZ$ are linear operators, 
and $c\in\cZ$ is a given vector. 
The alternating direction method of multipliers (ADMM), originally proposed by Glowinski and Marroco \cite{glowinski1975approximation} and Gabay and Mercier \cite{gabay1976dual}, stands out as one of the most widely used methods for solving \eqref{problem 1}, especially in distributed optimization and statistical learning \cite{boyd2011distributed,chang2014multi}. 

Given $\lambda >0$ as the penalty parameter, the augmented  Lagrangian function of \eqref{problem 1} is defined by 
\begin{align*}
\mathcal{L}_\lambda(x,y,z):=f(x)+g(y)-\langle Ax+By-c,z\rangle+\frac{\lambda}{2}\|Ax+By-c\|^{2},\qquad &\\
 x\in\cX,\, y\in\cY,\, z\in\cZ.&
\end{align*}
The ADMM minimizes $\mathcal{L}_\lambda(x,y,z)$ with respect to $x$ and $y$ successively, and then updates the dual variable in the same way as the method of multipliers \cite{hestenes1969multiplier,rockafellar1976monotone}. 
As illustrated by a counterexample in \cite{chen2017note}, the subproblems in ADMM may not be solvable.
Therefore, the well-definedness of ADMM should not be taken as guaranteed, and a natural remedy to this issue is to add proximal terms to the subproblems.
Initially, positive definite proximal terms were considered \cite{eckstein1994some,he2002new,lin2011linearized,yang2013linearized}. 
A much more practical variant of proximal ADMM was developed in \cite[Appendix B]{fazel2013hankel}, in which positive semi-definite proximal terms are sufficient, which provides crucial flexibility for choosing the proximal terms,  especially when combined with the symmetric Gauss-Seidel (sGS) decomposition technique \cite{li2019block} to solve multi-block problems \cite{li2016schur,chen2017efficient,chen2018unified,chen2021}.

Even if the subproblems are well defined, the classic ADMM  (with the unit dual step length) 
can only achieve an $O(1/K)$ ergodic (by measuring the average of the iterations) convergence rate \cite{monteiro2013iteration}, and an $O(1/\sqrt{K})$ non-ergodic convergence rate on primal feasibility violations and the gap of the primal objective value  \cite{davis2016convergence}.
Moreover, \cite{Cui2016ADMM} demonstrated that a majorized ADMM, including the classical ADMM, exhibits a non-ergodic $O(1/\sqrt{K})$ convergence rate concerning the Karush-Kuhn-Tucker (KKT) optimality condition.
To accelerate ADMM by extrapolation, \cite{goldstein2014fast} realized the Nesterov acceleration method when the objective functions are strongly convex. 
Furthermore, \cite{ouyang2015accelerated} obtained an ergodic convergence rate superior to $O(1/K)$ for linearized ADMM (including ADMM) by assuming the Lipschitz continuous differentiability of a function in the objective. 
For the classic ADMM with unit step length, \cite{xu2017accelerated} achieved an ergodic convergence rate of $O(1/K^{2})$ by properly tuning the penalty parameter when one of the objective functions is strongly convex.  
In addition, \cite{deng2017parallel} obtained a convergence rate of  $o(1/\sqrt{K})$  for the weighted distance between successive iterations. 
More recently, \cite[Theorem 3.3]{zhang2023lagrangian} extended the result of \cite{xu2017accelerated} by allowing the step length in $(0, (1+\sqrt{5})/2]$.

Theoretically, the ergodic convergence rates for the accelerated variants of the ADMM are satisfactory.
From a practical perspective, ergodic averaging might compromise the properties of sparseness and low rank in sparse and low-rank optimization problems.
Specifically, in image processing and low-rank learning, it is typically crucial to maintain the low-rankness and sparseness of iterative sequences. 
For example, if objective functions are designed to promote sparsity, the tail of the iteration sequence exhibits sparsity, but the average is likely to be dense.
To overcome this limitation, accelerating the non-ergodic convergence rate of ADMM has been considered. 

\cite{li2019accelerated} accelerated ADMM and obtained a non-ergodic $O(1/K)$  convergence rate in terms of function values and feasibility violations, when the dual step length is restricted in $(0.5,1)$.
It was shown in \cite[Section 4]{li2019accelerated} that such a non-ergodic convergence rate is tight, provided that the algorithm is well defined.
For proximal ADMM with the proximal terms for one variable being positive definite, 
\cite{sabach2022faster}
introduced an accelerated proximal ADMM with a non-ergodic convergence rate of $O(1/K)$ in terms of function values and feasibility violations. 
Moreover, by reformulating ADMM as a fixed-point iterative method,
\cite{kim1905accelerated} obtained an accelerated ADMM that possesses a non-ergodic $O(1/K)$ convergence rate regarding primal feasibility violations.
Very recently, based on the findings in \cite{zhang2022efficient} and \cite{yang2025accelerated},
\cite{sun2024accelerating} proposed an accelerated semi-proximal ADMM (with asymptotic $o(1/k)$ and non-asymptotic $O(1/k)$ convergence rates concerning the KKT residual and the primal objective function value gap) by establishing the equivalence between the preconditioned ADMM and the degenerate proximal point algorithm \cite{bonnans95,li2020}, and proposing an accelerated degenerate proximal point algorithm using the Halpern iteration and the fast Krasnosel'ski{\u{\i}}-Mann iteration. 
Subsequently, a linear programming solver based on \cite{sun2024accelerating} was developed in \cite{chen2024HPR}, and one may refer to \cite[Section 1]{sun2024accelerating} for a comprehensive introduction to the accelerated variants of ADMM. 
In particular, similar to \cite[Appendix B]{fazel2013hankel}, the proximal terms in \cite{sun2024accelerating} are allowed to be positive semi-definite. This, together with the excellent effect of the method in \cite{li2019accelerated} for sparse optimization, motivates us to consider accelerating the sPADMM using Nesterov’s extrapolation as \cite{li2019accelerated}.

The results in \cite{li2019accelerated} are quite appealing, especially the non-ergodic accelerated convergence rate of the following Algorithm \ref{algg1}, which is favorable for sparse and low-rank optimization problems. 
However, the subproblems of Algorithm \ref{algg1}, similar to the classic ADMM, may not be solvable \cite{chen2017note}. 
Therefore, to improve the robustness of Algorithm \ref{algg1} and make it more applicable to practical problems, it is natural to consider adding proximal terms to the subproblems. 
Consequently, how to properly add proximal terms to Algorithm \ref{algg1} to ensure that the subproblems are solvable while maintaining the accelerated ${O}(1/K)$ non-ergodic convergence rate comes to an important problem.

\begin{algorithm} 
\caption{Accelerated ADMM  for solving problem \eqref{problem 1}}
\label{algg1}
\KwIn{
$y^{-1}=y^{0}\in\cY$, $z^{0}\in\cZ$, $\lambda>0$, $\tau \in (0.5,1)$,  and $\theta^{-1}=1/\tau$.}
\KwOut{ $\{(x^k,y^k,z^k)\}$.} 
\For{\textnormal{$k=0,1, \ldots$,}}{
1. $\theta^{k}:=\frac{1}{1-\tau+1/\theta^{k-1}}$
\ $v^{k}:=y^{k}+\frac{\theta^{k}
(1-\theta^{k-1})}{\theta^{k-1}}(y^{k}-y^{k-1})$;
\\[1mm]
2. $x^{k+1}\in \underset{x}{\arg\min}~\mathcal{L}_{\frac{\lambda}{\theta^{k}}}(x,v^{k},z^{k});$
\\[1mm]
3. $y^{k+1}\in \underset{y}{\arg\min}~\mathcal{L}_{\frac{\lambda}{\theta^{k}}}(x^{k+1},y,z^{k})$;
\\[1mm]
4. $z^{k+1}:=z^{k}-\tau\lambda(Ax^{k+1}+By^{k+1}-c)$.
}
\end{algorithm}

In this paper, to address the problem mentioned above, we propose an accelerated semi-proximal alternating direction method of multipliers (AsPADMM), involving extrapolation techniques and increasing penalty parameters as Algorithm \ref{algg1}, and prove its non-ergodic convergence rate of $O(1/K)$. 
 Unlike Algorithm \ref{algg1},  AsPADMM can ensure that subproblems are always well defined by adding suitable proximal terms. 
Furthermore, the proposed AsPADMM with variable proximal terms, together with sGS decomposition techniques \cite{li2019block}, can properly solve multi-block non-smooth convex problems that frequently occur in image processing and low-rank learning.

The resulting algorithm, referred to as sGS-AsPADMM, also admits an $O(1/K)$ non-ergodic convergence rate. 

Finally, we apply the proposed sGS-AsPADMM to solving the multi-block subproblems in difference-of-convex (DC) algorithms for robust low-rank tensor completion problems and mixed sparse optimization problems, 
which come from image processing and low-rank learning. 

The numerical results validate the effectiveness of the acceleration in the sGS-AsPADMM.

The remaining parts of this paper are organized as follows. 
In Section \ref{sect:preliminary}, we provide the basic definitions and some preliminary results. 
In Section \ref{sect 2}, we analyze the $O(1/\sqrt{K})$ non-ergodic convergence rate of the sPADMM. 
In Section \ref{sect 3}, we propose the accelerated sPADMM algorithm and derive its non-ergodic $O(1/K)$  convergence rate. 
In Section \ref{sect 4}, we extend the proposed accelerated sPADMM algorithm to multi-block problems. 
In Section \ref{sect 5}, we provide the applications of the proposed algorithms to different applications, including the subproblems of non-convex robust low-rank tensor completion problems and the subproblems of non-convex mixed sparse optimization problems. 
We conclude the paper in Section \ref{sect 6}.

\section{Preliminaries}
\label{sect:preliminary}
This section gives some definitions and lemmas used throughout this paper. 
Let $\mathcal{S}:\cX\to\cX$ be any self-adjoint positive semi-definite linear operator. 
For any $x\in \cX$, define $\|x\|_{\mathcal{S}}:=\sqrt{\langle x,\mathcal{S}x\rangle}$. 
For any convex set $C\subseteq\cX$, its relative interior \cite[Section 2]{rockafellar1970convex} is denoted by ${\rm ri}\, C$, and we use
$\delta_{C}(\cdot)$ to denote the corresponding indicator function.
Moreover, for a convex function $f:\cX\to(-\infty, \infty]$, we denote its sub-differential mapping \cite[Section 23]{rockafellar1970convex} as $\partial f(\cdot)$ and use ${\rm dom}(f)$ to represent its domain.

Note that for the convex functions $f$ and $g$ in problem \eqref{problem 1}, there exist two self-adjoint and positive semi-definite linear operators $\Sigma_{f}:\cX\to \cX$ and $\Sigma_{g}:\cY\to\cY$ such that for all $x,~\hat{x}\in \text{dom}(f)$, $w\in \partial f(x)$ and $\hat{w}\in \partial f(\hat{x})$,
\begin{align}\label{budengshi1}
f(x)\geq f(\hat{x})+\langle \hat{w},x-\hat{x}\rangle+\frac{1}{2}\|x-
\hat{x}\|_{\Sigma_{f}}^{2}~~\text{and}~~\langle w-\hat{w},x-\hat{x}
\rangle \geq \|x-\hat{x}\|^{2}_{\Sigma_{f}},
\end{align}
and for all $y,~\hat{y}\in dom(g)$, $v\in \partial g(y)$ and $\hat{v}\in \partial g(\hat{y})$,
\begin{align}\label{budengshi2}
g(y)\geq g(\hat{y})+\langle \hat{v},y-\hat{y}\rangle+\frac{1}{2}\|y-
\hat{y}\|_{\Sigma_{g}}^{2}~~\text{and}~~\langle v-\hat{v},y-\hat{y}
\rangle \geq \|y-\hat{y}\|^{2}_{\Sigma_{g}}.
\end{align}
Moreover, 
the KKT system of problem \eqref{problem 1} is given by
\begin{equation}\label{kkt}
    Ax+By=c, \quad 0\in \partial f(x)-A^{*}z\quad \text{and} \quad 0\in \partial g(y)-B^{*}z.
\end{equation}
If $(x^{*},y^{*},z^{*})\in\cX\times \cY\times \cZ$ satisfies \eqref{kkt}, from \cite[Corollary 30.5.1]{rockafellar1970convex} we know that $(x^{*},y^{*})$ is an
optimal solution to the problem \eqref{problem 1} and $z^{*}$ is an optimal solution to the dual of this problem.
In this case, it is easy to see from \eqref{kkt} that
\begin{equation}
\label{opteq}
f(x)+g(y)\ge f(x^*)+g(y^*)+\langle z^*, Ax+B y-c\rangle. 
\end{equation}

Finally, we prepare the following lemma for our discussion. 
\begin{lemma}\label{lemma 2} \cite[Lemma 2]{li2019accelerated}
Given the sequence $\{\nu^{k}\}\subset\cZ$ satisfying 
$$
\|(1+K(1-\tau))\nu^{K+1}+\tau\sum_{k=1}^{K}\nu^{k}\|\leq \tau M, \quad \forall K=0,1,2...,
$$
where $\tau\in(0,1)$. 
Then $\|\sum_{k=1}^{K}\nu^{k}\|\leq M$  for all $ K=1,2,\cdots$.
\end{lemma}

\section{Convergence rate of semi-proximal ADMM}\label{sect 2}

In this section, we briefly introduce the sPADMM algorithm (given as Algorithm \ref{alg1}) proposed in \cite[Appendix B]{fazel2013hankel} and its convergence. Then analyze its convergence rate. 

\begin{algorithm}
\caption{sPADMM algorithm for solving \eqref{problem 1}}\label{alg1}
\KwIn{$x^{0}\in\cX$, $y^{0}\in\cY$, $z^{0}\in\cZ$, $\lambda>0$, $\tau>0$. Choose self-adjoint linear operators $\mathcal{S}:\cX \to\cX$ and $\mathcal{T}:\cY \to\cY$ such that
$\mathcal{S}\succeq 0$, $\mathcal{T} \succeq 0$, $\Sigma_{f}+\mathcal{S}+\lambda A^{*}A\succ 0$, and $ \Sigma_{g}+\mathcal{T}+\lambda B^{*}B\succ 0$.}
\KwOut{$\{(x^k,y^k,z^k)\}$.} 
\For{$k=0,1, \ldots$,}{
1. $ x^{k+1}:=\underset{x}{\arg\min}~\mathcal{L}_{\lambda}(x,y^{k},z^{k})+\frac
{1}{2}\|x-x^{k}\|^{2}
_{\mathcal{S}}$;\\
2. $ y^{k+1}:=\underset{y}{\arg\min}~\mathcal{L}_{\lambda}(x^{k+1},y,z^{k})
+\frac{1}{2}\|y-y^{k}\|^{2}
_{\mathcal{T}}$;\\
3. $ z^{k+1}:=z^{k}-\tau\lambda(Ax^{k+1}+By^{k+1}-c)$.}
\end{algorithm}

Note that Algorithm \ref{alg1} includes the classical ADMM \cite{glowinski1975approximation,gabay1976dual} as a special case by taking $\mathcal{S}=0$ and $\mathcal{T}=0$.
It also covers the proximal ADMM in \cite{eckstein1994some} in which $\cS$ and $\cT$ are the identity operators. 
According to \cite[Theorem B.1]{fazel2013hankel}, when the solution set of \eqref{problem 1} is not empty and there exists $(x,y) \in {\rm ri}({\rm dom}(f)\times {\rm dom} (g))\cap P$, where ${\rm P}$ is the constraint set of \eqref{problem 1}, 
the sequence $\{(x^{k},y^{k})\}$ converges to an optimal solution of \eqref{problem 1}, and $\{z^{k}\}$ converges to an optimal solution to the dual problem of \eqref{problem 1} if $\tau\in(0,\frac{1+\sqrt{5}}{2})$.
In fact, the first two conditions can be replaced by the condition that the solution set of the KKT system \eqref{kkt} is non-empty \cite[Corollaries 28.2.2]{rockafellar1970convex}.

To see if Algorithm \ref{alg1} can be accelerated using the extrapolation techniques in \cite{han2018linear}, we should focus on its non-ergodic convergence rate. 
To this end, \cite{Cui2016ADMM} demonstrated that a majorized ADMM, including the classical ADMM, exhibits a non-ergodic $O(1/\sqrt{K})$ convergence rate with respect to the optimality condition of KKT. 
Note that if $(x^{*},y^{*},z^{*})$ is a solution to the KKT system \eqref{kkt} of the problem \eqref{problem 1}, one has $Ax^{*}+By^{*}=c$ and 
$$
f(x)+g(y)-f(x^{*})-g(y^{*})+\langle z^{*},c-Ax-By \rangle \geq 0\quad   \forall x\in \cX, \, y\in\cY.
$$
Based on this observation, \cite[Theorems 13(2) \& 15(1)]{davis2016convergence} characterize the convergence rate based on the primal feasibility and the gap from the objective value of iterations to the optimal value. Using the same strategy to characterize the convergence rate, we give the following result, an extension of the result in \cite{davis2016convergence} and the corresponding proof is given in Appendix \ref{appendix1}.

\begin{theorem}\label{thm1}
Assume that the solution set of the KKT system \eqref{kkt} is non-empty. 
Let $\{(x^{k},y^{k},z^{k})\}$ be the sequence generated by Algorithm \ref{alg1} with $\tau\in(0,1]$. Assume that $\{(x^{*},y^{*})\}$ is an optimal solution of the problem \eqref{problem 1}, and $\{z^{*}\}$ is an optimal solution to the dual problem of problem \eqref{problem 1}. 
Then, for all $K= 1,2,3\ldots$, one has
$$
\begin{cases}
\begin{array}{r}
-\|z^{*}\|\sqrt{\frac{C}{\tau\lambda K}}
\leq f(x^{K+1})+g(y^{K+1})-f(x^{*})-g(y^{*})
\qquad~
\\ 
\leq \|z^{*}\|\sqrt{\frac{C}{\tau\lambda K}}+\frac{4C}{\sqrt{K}}+\frac{C}{K\sqrt{\tau}},
\end{array}
\\[2mm]
0\leq\|Ax^{K+1}+By^{K+1}-c\|
=\frac{1}{\tau\lambda}\|z^{K+1}-z^{K}\|\leq \sqrt{\frac{ C}{\tau\lambda K}}, 
\end{cases}
$$
where $C=\max\{C_{1},C_{2}\}$ is a nonnegative constant with 
$$
C_{1}:=\frac{1}{\lambda}\|z^{1}-z^{*}\|^{2}+\lambda\|B(y^{1}-y^{*})\|^{2}+
\|x^{1}-x^{*}\|_{\mathcal{S}}^{2}+(
\|y^{1}-y^{*}\|_{\mathcal{T}}^{2}+
\|y^{0}-y^{1}\|_{\mathcal{T}}^{2}),$$  
and
$C_{2}:= \frac{1}{m}
\left(\frac{1}{\tau\lambda}\|z^{1}-z^{*}\|^{2}+\lambda\|B(y^{1}-y^{*})\|^{2}+
\|x^{1}-x^{*}\|_{\mathcal{S}}^{2}+
\|y^{1}-y^{*}\|_{\mathcal{T}}^{2}
\right).
$
\end{theorem}

\begin{remark}
Theorem \ref{thm1} implies that the sequence $\{f(x^{k})+ g(y^{k})\}$ converges to the optimal value of problem \eqref{problem 1} with
the non-ergodic convergence rate $O(1/ \sqrt{K})$.
We should mention that when some error bound conditions are imposed on the KKT system \eqref{kkt} of the problem \eqref{problem 1}, one can further obtain a linear convergence rate of Algorithm \ref{alg1} as in \cite{han2018linear}. 
Moreover, since the $O(1/\sqrt{K})$ non-ergodic convergence rate for the classical ADMM, a special case of sPADMM, was shown to be optimal \cite{davis2016convergence,liu2019linearized}, 
the $O(1/\sqrt{K})$ non-ergodic convergence rate of Algorithm \ref{alg1} for solving problem \eqref{problem 1} in Theorem \ref{thm1} is also optimal.
\end{remark}

\section{An accelerated semi-proximal ADMM}\label{sect 3}
Based on introducing extrapolation techniques and increasing the penalty parameter to Algorithm \ref{alg1}, we
propose the following accelerated semi-proximal ADMM (AsPADMM) algorithm to solve problem \eqref{problem 1}. 
\begin{algorithm}
\caption{AsPADMM algorithm for solving problem \eqref{problem 1}}
\label{alg2}
\KwIn{
$x^{0}\in\cX$, $y^{0}\in\cY$, $z^{0}\in\cZ$, $y^{-1}=y^0$,
$\lambda>0$, $\tau\in(0,1)$, and $\theta^{-1}=1/\tau$. Choose self-adjoint positive semidefinite linear operators
$\mathcal{S}:\cX \to\cX$ and $\mathcal{T}:\cY \to\cY$ such that $\Sigma_{f}+\mathcal{S}+\lambda A^{*}A\succ 0$, and $ \Sigma_{g}+\mathcal{T}+\lambda B^{*}B\succ 0$.}
\KwOut{ $\{(x^k,y^k,z^k)\}$.} 
\For{$k=0,1, \ldots$,}{
1. $\theta^{k}:=\frac{1}{1-\tau+1/\theta^{k-1}}$,\ 
$ v^{k}:=y^{k}+\frac{\theta^{k}(1-\theta^{k-1})}{\theta^{k-1}}(y^{k}-y^{k-1})$;\\
2. $ x^{k+1}:=\underset{x}{\arg\min}~\mathcal{L}_\frac{\lambda}{\theta^{k}}(x,v^{k},z^{k})+\frac{1}{2}\|x-x^{k}\|^{2}_{\mathcal{S}}$;\\
3. $ y^{k+1}:=\underset{y}{\arg\min}~\mathcal{L}_\frac{\lambda}{\theta^{k}}(x^{k+1},y,z^{k})+\frac{1}{2}\|y-y^{k}\|^{2}_{\mathcal{T}}$;\\
4.   $ z^{k+1}:=z^{k}-\tau\lambda(Ax^{k+1}+By^{k+1}-c)$.}
\end{algorithm}

In order to generalize Algorithm \ref{alg2} to solve multi-block problems using the sGS decomposition techniques, we propose the following Algorithm \ref{algst}, which takes Algorithm \ref{alg2} as its special case. 

\begin{algorithm} 
\caption{AsPADMM algorithm with variable proximal terms}
\label{algst}
\KwIn{
$x^{0}\in\cX$, $y^{0}\in\cY$, $z^{0}\in\cZ$, $y^{-1}=y^0$,
$\lambda>0$, $\tau\in(0,1)$, and $\theta^{-1}=1/\tau$. 
Choose self-adjoint positive semidefinite linear operators
$\mathcal{S}^{k}:\cX \to\cX$ and 
$\mathcal{T}^{k}:\cY \to\cY$. 
}
\KwOut{ $\{(x^k,y^k,z^k)\}$.} 
\For{$k=0,1, \ldots$,}{
1. $\theta^{k}:=\frac{1}{1-\tau+1/\theta^{k-1}}$,\ 
$ v^{k}:=y^{k}+\frac{\theta^{k}(1-\theta^{k-1})}{\theta^{k-1}}(y^{k}-y^{k-1})$;\\
2. $ x^{k+1}:=\underset{x}{\arg\min}~\mathcal{L}_\frac{\lambda}{\theta^{k}}(x,v^{k},z^{k})+\frac{1}{2}\|x-x^{k}\|^{2}_{\mathcal{S}^{k}}$;\\
3. $ y^{k+1}:=\underset{y}{\arg\min}~\mathcal{L}_\frac{\lambda}{\theta^{k}}(x^{k+1},y,z^{k})+\frac{1}{2}\|y-y^{k}\|^{2}_{\mathcal{T}^{k}}$;\\
4.   $ z^{k+1}:=z^{k}-\tau\lambda(Ax^{k+1}+By^{k+1}-c)$.}
\end{algorithm}

For Algorithm \ref{algst}, we have the following result regarding its convergence and its $O(1/{K})$ non-ergodic convergence rate in terms of the objective value and the violation of the feasibility.  
\begin{theorem}
\label{thmst}
Let $(x^{*},y^{*},z^*)$ be a solution to the KKT system \eqref{kkt} of problem \eqref{problem 1}, and let $\{(x^{k},y^{k},z^{k})\}$ be the sequence generated by Algorithm \ref{algst} 
with $\Sigma_{f}+A^{*}A+\mathcal{S}^{k}\succ 0$ and $\Sigma_{g}+B^*  B+\mathcal{T}^{k}\succ 0$ for all $k\ge 0$. 
Suppose that $\Sigma_{f}\succeq \cS^{k+1}-\cS^{k}$ and $\Sigma_{g}\succeq \cT^{k+1}-\cT^{k}$ for all $k\ge 0$.
Then, for all $K= 0,1,2\ldots$, one has
\begin{subequations}
\begin{equation}
-\frac{2C_{3}\|z^{*}\|}{1+K(1-\tau)}\leq f(x^{K+1})+g(y^{K+1})-f(x^{*})-g(y^{*})\leq \frac{2C_{3}\|z^{*}\|+C_{4}}{1+K(1-\tau)},
\label{con2}
\end{equation}   
and
\begin{equation}
0\leq \|c-Ax^{K+1}-By^{K+1}\|\leq \frac{2C_{3}}{1+K(1-\tau)},
\label{kuoda62}
\end{equation}
\end{subequations}
where $C_{3}$ and $C_{4}$ are fixed nonnegative real numbers defined by 
$$
C_{3}:=\frac{2}{\lambda}\|z^{0}-z^{*}\|+\|By^{0}-By^{*}\|+\frac{1}{\sqrt{\lambda}}(\|x^{0}-x^{*}\|_{\cS^{0}}+\|y^{0}-y^{*}\|_{\cT^{0}}),
$$
and 
$$
C_{4}:=\frac{1}{2\lambda}\|z^{0}-z^{*}\|^{2}+\frac{\lambda}{2}
\|By^{0}-By^{*}\|^{2}+\frac{1}{2}(\|x^{0}-x^{*}\|^{2}_{\cS^{0}}+
\|y^{0}-y^{*}\|^{2}_{\cT^{0}}).
$$
\end{theorem}

\begin{proof}
For convenience, denote
\begin{equation}
\label{dingyi3}
\begin{cases}
\hat{z}^{k}:=z^{k}+\frac{\lambda(1-\theta^{k})}{\theta^{k}}
(c-Ax^{k}-By^{k}),
&k=0,1,\ldots
,\\[.5mm]
\overline{z}^{k+1}:=z^{k}+\frac{\lambda}{\theta^{k}}
(c-Ax^{k+1}-Bv^{k}),
&k=0,1,\ldots,
\\[.5mm]
\hat{y}^{k+1}:=\frac{1}{\theta^{k}}y^{k+1}-\frac{1-\theta^{k}}{\theta^{k}}
y^{k},
&k=-1,0,1,\ldots. 
\end{cases}
\end{equation}
It is easy to see from Algorithm \ref{alg2} that $\theta^{0}=1$. 
Then, from \eqref{dingyi3} one has $\hat{z}^{0}=z^{0}$, $\hat{y}^{0}=y^{0}$, and $\hat{y}^1=y^1$. 
Also, from the updating rule of $\theta^k$ in Algorithm \ref{alg2} one has for all $k\ge 0$
\begin{equation}
\label{thetarelation}
{\frac{1}{\theta^{k}}-1+\tau}=\frac{1}{\theta^{k-1}}
\quad 
\mbox{and}
\quad
\theta^{k}=\frac{1}{k(1-\tau)+1}.
\end{equation}
Then, one can see that for any $k\ge 0$, 
\begin{equation}
\label{usefull2}
\begin{array}{lll}
&\hat{z}^{k+1}-\overline{z}^{k+1}
\\[1.5mm]
=&z^{k+1}+\frac{\lambda(1-\theta^{k+1})
}{\theta^{k+1}}(c-Ax^{k+1}-By^{k+1})-z^{k}-\frac{\lambda}{\theta^{k}}
(c-Ax^{k+1}-Bv^{k})
\\[1.5mm]
=&z^{k}+ (\tau+\frac{1}{\theta^{k+1}}-1)\lambda(c-Ax^{k+1}-By^{k+1})-z^{k}-\frac{\lambda}{\theta^{k}}
(c-Ax^{k+1}-Bv^{k})
\\[1.5mm]
=&z^{k}+\frac{\lambda}{\theta^{k}}(c-Ax^{k+1}-By^{k+1})
-z^{k}-\frac{\lambda}{\theta^{k}}
(c-Ax^{k+1}-Bv^{k})
\\[1.5mm]
=&\frac{\lambda}{\theta^{k}}(Bv^{k}-By^{k+1}),
\end{array}
\end{equation}
and 
\begin{equation}
\label{usefull3}
\begin{array}{ll}
v^{k}-(1-\theta^{k})y^{k}
&=y^{k}+\frac{\theta^{k}
(1-\theta^{k-1})}{\theta^{k-1}}(y^{k}-y^{k-1})-(1-\theta^{k})y^{k}
\\[1.5mm]
&
=\frac{\theta^{k}}{\theta^{k-1}}y^{k}-\frac{\theta^{k}(1-\theta^{k-1
})}{\theta^{k-1}}y^{k-1}=\theta^{k}\hat{y}^{k}.
\end{array}
\end{equation}
Meanwhile, it holds for all $k\ge 0$ that
\begin{equation}
\label{usefull1}
\hat{z}^{k+1}-\hat{z}^{k}=\frac{\lambda}{\theta^{k}}(c-Ax^{k+1}-By^{k+1})
-(\frac{\lambda}{\theta^{k-1}}-\tau \lambda)(c-Ax^{k}-By^{k}).
\end{equation}
For all $k\ge 0$ one has from the optimality conditions of the subproblems in Algorithm \ref{algst} that
\begin{equation*}
\begin{array}{ll}
u^{k+1}:&=A^{*}z^{k}-\frac{\lambda}{\theta^{k}}A^{*}(Ax^{k+1}
+Bv^{k}-c)-\mathcal{S}^{k}(x^{k+1}-x^{k}), 
\\[1.5mm]
&=A^{*}\overline{z}^{k+1}-
\mathcal{S}^{k}(x^{k+1}-x^{k})\in \partial f(x^{k+1}),
\\[1.5mm]
w^{k+1}:&=B^{*}z^{k}-\frac{\lambda}{\theta^{k}}B^{*}(Ax^{k+1}
+By^{k+1}-c)-\mathcal{T}^{k}(y^{k+1}-y^{k})
\\[1.5mm]
&=B^{*}\overline{z}^{k+1}
-\frac{\lambda}{\theta^{k}}B^{*}B(y^{k+1}-v^{k})-\mathcal{T}^{k}(y^{k+1}-y^{k})\in \partial g(y^{k+1}). 
\end{array}
\end{equation*}
Then, by combining \eqref{budengshi1}, \eqref{budengshi2}, \eqref{kkt} and \eqref{kuoda1} one has for any $x\in\cX$ and $y\in\cY$, 
\begin{equation}
\label{kuoda1}
\begin{array}{lll}
&f(x^{k+1})+g(y^{k+1})-f(x)-g(y)
+\frac{1}{2}\|x^{k+1}-x\|_{\Sigma_{f}}^{2}+\frac{1}{2}\|y^{k+1}-y\|_{\Sigma_{g}}^{2}
\\[1.5mm]
\leq &\langle u^{k+1},x^{k+1}-x\rangle+\langle w^{k+1},
y^{k+1}-y\rangle
\\[1.5mm]
=&\langle \overline{z}^{k+1},Ax^{k+1}+By^{k+1}-Ax-By\rangle
-\frac{\lambda}{\theta^{k}}\langle B(y^{k+1}-v^{k}),B(y^{k+1}-y)\rangle
\\[1.5mm]
&-\langle x^{k+1}-x^{k},\mathcal{S}^{k}(x^{k+1}-x)\rangle-\langle y^{k+1}-y^{k},\mathcal{T}^{k}(y^{k+1}-y)\rangle.
\end{array}
\end{equation}
Note that $\frac{1}{\theta^{k-1}}-\tau=\frac{1-\theta^{k}}{\theta^{k}}$ for all $k\ge 0$. Therefore, by using \eqref{usefull1} one has
\begin{equation}
\label{defvarsigma}
\begin{array}{rll}
\varsigma^k:=&\frac{1}{\theta^{k}}[f(x^{k+1})+g(y^{k+1})-f(x^{*})-g(y^{*})+\langle z^{*},c-Ax^{k+1}-By^{k+1}\rangle]
\\[1.5mm]
&-(\frac{1}{\theta^{k-1}}-\tau)[f(x^{k})+g(y^{k})
-f(x^{*})-g(y^{*})+\langle z^{*},c-Ax^{k}-By^{k}\rangle]
\\[1.5mm]
=&\frac{1}{\lambda}\langle z^{*},\hat z^{k+1}-\hat z^k\rangle+
\frac{1-\theta^{k}}{\theta^{k}}[f(x^{k+1})+g(y^{k+1})-f(x^{k})-g({y^{k}})]
\\[1.5mm]
&+\big[f(x^{k+1})+g(y^{k+1})-f(x^{*})-g({y^{*}})\big]
\ge 0.
\end{array}
\end{equation}
One can take $(x,y)=(x^{*},y^{*})$ and $(x,y)=(x^{k},y^{k})$ in \eqref{kuoda1} to get two inequalities, based on which and \eqref{usefull1} one can see from the above equality that 
\begin{equation}
\label{kuoda20}
\begin{array}{ll}
\varsigma^k \leq
\frac{1}{\lambda}\langle z^{*},\hat z^{k+1}-\hat z^k\rangle+
\frac{1-\theta^{k}}{\theta^{k}}\langle \overline{z}^{k+1},Ax^{k+1}+By^{k+1}-c-(Ax^k+By^k-c)\rangle
\\[1.5mm]
\quad +\langle \overline{z}^{k+1},Ax^{k+1}+By^{k+1}-c\rangle
-\frac{(1-\theta^{k})\lambda}{(\theta^{k})^2}\langle B(y^{k+1}-v^{k}),B(y^{k+1}-y^k)\rangle
\\[1.5mm]
\quad -\frac{\lambda}{\theta^{k}}\langle B(y^{k+1}-v^{k}),B(y^{k+1}-y^*)\rangle
-\frac{1}{2}(\|x^{k+1}- {x}^*\|^{2}_{\Sigma_{f}}
+\|y^{k+1}- {y}^*\|^{2}_{\Sigma_{g}})
\\[1.5mm]
\quad -\frac{1-\theta^{k}}{\theta^{k}}\langle x^{k+1}-x^{k},\cS^{k}(x^{k+1}-x^{k})\rangle 
-\langle x^{k+1}-x^{k},\mathcal{S}^{k}(x^{k+1}-x^*)\rangle
\\[1.5mm]
\quad -\frac{1-\theta^{k}}{\theta^{k}}
\langle y^{k+1}-y^{k},\cT^{k}(y^{k+1}-y^{k})\rangle
-\langle y^{k+1}-y^{k},\mathcal{T}^{k}(y^{k+1}-y^*)\rangle
\\[1.5mm]
= \frac{1}{\lambda}\langle z^{*}-\overline{z}^{k+1},
\hat{z}^{k+1}-\hat{z}^{k}\rangle
-\frac{1}{2}\|x^{k+1}- {x}^*\|^{2}_{\Sigma_{f}}
-\frac{1}{2}\|y^{k+1}- {y}^*\|^{2}_{\Sigma_{g}}
\\[1.5mm]
\quad -\frac{\lambda}{(\theta^{k})^{2}}\langle B(y^{k+1}-v^{k}),B(y^{k+1}-\theta^{k}y^{*}-(1-\theta^{k})y^{k})\rangle
\\[1.5mm]
\quad -\frac{1-\theta^{k}}{\theta^{k}}\|x^{k+1}-x^{k}\|^{2}_{\mathcal{S}^{k}}
-\langle x^{k+1}-x^{k},\mathcal{S}^{k}(x^{k+1}-x^*)\rangle
\\[1.5mm]
\quad 
-\frac{1-\theta^{k}}{\theta^{k}}\|y^{k+1}
-y^{k}\|^{2}_{\mathcal{T}^{k}}
-\langle y^{k+1}-y^{k},\mathcal{T}^{k}(y^{k+1}-y^*)\rangle
\\[1.5mm]
\leq \frac{1}{2\lambda}[\|\hat{z}^{k}-z^{*}\|^{2}-\|\hat{z}^{k+1}-z^{*}\|^{2}
+\|\overline{z}^{k+1}-\hat{z}^{k+1}\|^{2}
-\|\overline{z}^{k+1}-\hat{z}^{k}\|^{2}]
\\[1.5mm]
\quad 
-\frac{1}{2}\|x^{k+1}- {x}^*\|^{2}_{\Sigma_{f}}
-\frac{1}{2}\|y^{k+1}- {y}^*\|^{2}_{\Sigma_{g}}
\\[1.5mm]
\quad
+\frac{\lambda}{2(\theta^{k})^{2}} \big(\|B(v^{k}-\theta^{k}y^{*}
-(1-\theta^{k})y^{k})\|^{2}
-\|B(y^{k+1}-v^{k})\|^{2}
\\[1.5mm]
\quad -\|B(y^{k+1}-\theta^{k}y^{*}-(1-\theta^{k})y^{k})\|^{2}\big)
\\[1.5mm]
\quad
-\frac{1}{2}\big (\|x^{k+1}-x^*\|_{\cS^{k}}^{2}
-\|x^{k}-x^*\|_{\cS^{k}}^{2}+\|x^{k}-x^{k+1}\|^2_{\cS^{k}}\big)
\\[1.5mm]
\quad  
-\frac{1}{2}\big( \|y^{k+1}-y^*\|_{\cT^{k}}^{2}
-\|y^{k}-y^*\|_{\cT^{k}}^{2}+\|y^k-y^{k+1}\|_{\cT^{k}}^{2}\big).
\end{array}
\end{equation}
Invoking \eqref{dingyi3}, \eqref{usefull2} and \eqref{usefull3} into the above inequality \eqref{kuoda20} and using the fact that $\theta^k\le 1$ for all $k\ge 0$ one can get  
\begin{equation}
\label{kuoda2}
\begin{array}{ll}
\varsigma^k 
\leq
\frac{1}{2\lambda} 
(\|\hat{z}^{k}-z^{*}\|^{2}-
\|\hat{z}^{k+1}-z^{*}\|^{2}
-\|\overline{z}^{k+1}-\hat{z}^{k}\|^{2})
\\[1.5mm]
\quad +\frac{\lambda}{2}(\|B(\hat{y}^{k}-y^{*})\|^{2}-
\|B(\hat{y}^{k+1}-y^{*})\|^{2})
\\[1.5mm]
\quad +\frac{1}{2}(\|x^{k}-x^{*}\|_{\mathcal{S}^{k}}
^{2}-\|x^{k+1}-x^{*}\|_{\mathcal{S}^{k+1}}^{2}+\|y^{k}-
y^{*}\|_{\mathcal{T}^{k}}^{2}-\|y^{k+1}-y^{*}\|_{\mathcal{T}^{k+1}}^{2})
\\[1.5mm]
\quad -\frac{1}{2}(\|x^{k+1}-x^{k}\|_{\Sigma_{f}+\cS^{k}-\cS^{k+1}}+\|y^{k+1}-y^{k}\|_{\Sigma_{g}+\cT^{k}-\cT^{k+1}})\\[1.5mm]
\leq 
\frac{1}{2\lambda} 
(\|\hat{z}^{k}-z^{*}\|^{2}-
\|\hat{z}^{k+1}-z^{*}\|^{2}+\frac{\lambda}{2}(\|B(\hat{y}^{k}-y^{*})\|^{2}-
\|B(\hat{y}^{k+1}-y^{*})\|^{2})
\\[1.5mm]
\quad +\frac{1}{2}(
\|x^{k}-x^{*}\|_{\mathcal{S}^{k}}^{2}
-\|x^{k+1}-x^{*}\|_{\mathcal{S}^{k+1}}^{2}
+\|y^{k}-y^{*}\|_{\mathcal{T}^{k}}^{2}
-\|y^{k+1}-y^{*}\|_{\mathcal{T}^{k+1}}^{2}).
\end{array}
\end{equation}
Summing up the inequality \eqref{kuoda2} from $k = 0$ to $K\ge 0$ and using \eqref{defvarsigma} obtains
\begin{equation}
\label{kuoda3}
\begin{array}{rll}
0\leq&\frac{1}{\theta^{K}}(f(x^{K+1})+g(y^{K+1})-f(x^{*})-g(y^{*})+\langle z^{*},c-Ax^{K+1}-By^{K+1}\rangle
\\[1.5mm]
&+\tau\sum_{k=0}^{K}(f(x^{k})+g(y^{k})-f(x^{*})-g(y^{*})+\langle z^{*},c-Ax^{k}-By^{k}\rangle)
\\[1.5mm]
\leq 
&\frac{1}{2\lambda}(\|\hat{z}^{0}-z^{*}\|^{2}-\|\hat{z}^{K+1}-z^{*}\|^{2})
+\frac{\lambda}{2}(\|B\hat{y}^{0}-By^{*}\|^{2}-\|B\hat{y}^{K+1}-By^{*}\|
^{2})
\\[1.5mm]
&+\frac{1}{2}(\|x^{0}-x^{*}\|^{2}_{\mathcal{S}^{0}}+
\|y^{0}-y^{*}\|^{2}_{\mathcal{T}^{0}}\\[1.5mm]
&-\frac{1}{2}(\|x^{K+1}-x^{*}\|^{2}_{\mathcal{S}^{K+1}}+
\|y^{K+1}-y^{*}\|^{2}_{\mathcal{T}^{K+1}})
. 
\end{array}
\end{equation}
Consequently, 
one can get from \eqref{kuoda3} that for all $K\ge 0$, 
\begin{equation*}
\begin{array}{ll}
&\|\hat{z}^{K+1}-z^{*}\|
\\[1.5mm] 
\leq  &\sqrt{\|\hat{z}^{0}-z^{*}\|^{2}+\lambda^{2}
\|B\hat{y}^{0}-By^{*}\|^{2}+\lambda\|x^{0}-x^{*}\|^{2}_{\mathcal{S}^{0}}+
\lambda\|y^{0}-y^{*}\|^{2}_{\mathcal{T}^{0}}}
\\[1.5mm]
\leq & \|\hat{z}^{0}-z^{*}\|+\lambda
\|B\hat{y}^{0}-By^{*}\|+\sqrt{\lambda}\|x^{0}-x^{*}\|_{\mathcal{S}^{0}}+
\sqrt{\lambda}\|y^{0}-y^{*}\|_{\mathcal{T}^{0}}.
\end{array}
\end{equation*}
Thus, based on \eqref{dingyi3} and the above inequality one can get for all $K\ge 0$, 
\begin{equation}
\label{kuoda52}
\begin{array}{lll}
\|\frac{1}{\theta^{K}}(c-Ax^{K+1}-By^{K+1})+\tau\sum_{k=1}^{K}
(c-Ax^{k}-By^{k})\|
\\[1.5mm]
=
\| \frac{1-\theta^{k+1}}{\theta^{k+1}} (c-Ax^{K+1}-By^{K+1})+\tau\sum_{k=1}^{K+1}
(c-Ax^{k}-By^{k})\|
\\[1.5mm]
=\frac{1}{\lambda}\|\hat{z}^{K+1}-\hat{z}^{0}\|\leq
\frac{1}{\lambda}\|\hat{z}^{K+1}-\hat{z}^{*}\|+\frac{1}{\lambda}\|\hat{z}^{0}-\hat{z}^{*}\|
\leq 
C_{3}.
\end{array}
\end{equation}
Particularly, taking $K=0$ in \eqref{kuoda52} implies that $\|c-Ax^{K+1}-By^{K+1}\|\le C_3$. 
Moreover, taking $\nu^{k}:=c-Ax^{k}-By^{k}$ and using the second equality in \eqref{thetarelation} one can see from the above inequality that for any $K\ge 1$, 
\begin{equation}
\label{kuoda6}
\begin{aligned}
\|(1+K(1-\tau))\nu^{K+1}+\tau\sum_{k=1}^{K}\nu^{k}\|
=
\|\frac{1}{\theta^{K}}\nu^{K+1}+\tau\sum_{k=1}^{K}\nu^{k}\|
\leq C_{3}.
\end{aligned}
\end{equation}
Invoking Lemma \ref{lemma 2}, one can get $\|\sum_{k=1}^{K}\nu^{k}\|\leq \frac{1}{\tau}C_{3}$ for all $K\ge 1$.
Thus, from \eqref{kuoda6} one can see that for all $k\ge 1$, 
$$
\|\nu^{K+1}\|\le \frac{C_3}{(1+K(1-\tau))}+\frac{\tau}{(1+K(1-\tau))}\|\sum_{k=1}^{K}\nu^{k}\|
\le
\frac{2C_{3}}{1+K(1-\tau)}.
$$
Based on the above discussions one can see that \eqref{kuoda62} holds for all $k\ge 0$. 
Furthermore, by using \eqref{opteq}, one can see from \eqref{kuoda3} that
$$
\begin{array}{rl}
0\leq &\frac{1}{\theta^{K}}(f(x^{K+1})+g(y^{K+1})-f(x^{*})-g(y^{*})
+\langle z^{*},c-Ax^{K+1}-By^{K+1})\rangle
\\[1.5mm]
\leq &C_4:=\frac{1}{2\lambda}\|\hat{z}^{0}-z^{*}\|^{2}+\frac{\lambda}{2}
\|B\hat{y}^{0}-By^{*}\|^{2}\\[1.5mm]
&~~~~~~~~+\frac{1}{2}(\|x^{0}-x^{*}\|^{2}_{\mathcal{S}^{0}}+
\|y^{0}-y^{*}\|^{2}_{\mathcal{T}^{0}}).
\end{array}
$$
Combining the above inequality and \eqref{kuoda62}, one can get \eqref{con2} holds for all $K\ge 0$.  
This completes the proof. 
\end{proof}

\begin{remark}
Actually, the assumptions of Theorem \ref{thmst} can be substituted for the inner product form, which is $\langle x,(\Sigma_{f}+\cS^{k}-\cS^{k+1})x\rangle\geq 0$ and $\langle x,(\Sigma_{y}+\cT^{k}-\cT^{k+1})x\rangle\geq 0$.

\end{remark}

Note that when $\cS^{k}$ and $\cT^{k}$ are fixed for all $k\ge 0$, conditions $\Sigma_{f}\succeq\cS^{k+1}-\cS^{k}$ and $\Sigma_{g}\succeq \cT^{k+1}-\cT^{k}$ always hold for all $k\ge 0$. Thus, we have the following result for the convergence of Algorithm \ref{alg2}.

\begin{corollary}
Let $(x^{*},y^{*},z^*)$ be a solution to the KKT system \eqref{kkt} of problem \eqref{problem 1}, and let $\{(x^{k},y^{k},z^{k})\}$ be the sequence generated by Algorithm \ref{alg2}. 
Then, for all $K= 0,1,2\ldots$, one has
$$
\begin{cases}
-\frac{2C_{5}\|z^{*}\|}{1+K(1-\tau)}\leq f(x^{K+1})+g(y^{K+1})-f(x^{*})-g(y^{*})\leq \frac{2C_{5}\|z^{*}\|+C_{6}}{1+K(1-\tau)},
\\[1.5mm]
0\leq \|c-Ax^{K+1}-By^{K+1}\|\leq \frac{2C_{5}}{1+K(1-\tau)},
\end{cases}
$$
where
$C_{5}:=\frac{2}{\lambda}\|z^{0}-z^{*}\|+
\|By^{0}-By^{*}\|+\frac{1}{\sqrt{\lambda}}\|x^{0}-x^{*}\|_{\mathcal{S}}+
\frac{1}{\sqrt{\lambda}}\|y^{0}-y^{*}\|_{\mathcal{T}},
$  
and 
$
C_{6}:=\frac{1}{2\lambda}\|z^{0}-z^{*}\|^{2}+\frac{\lambda}{2}
\|By^{0}-By^{*}\|^{2}+\frac{1}{2}(\|x^{0}-x^{*}\|^{2}_{\mathcal{S}}+
\|y^{0}-y^{*}\|^{2}_{\mathcal{T}}).
$
\end{corollary}
\begin{proof}
This is an immediate consequence of Theorem \ref{thmst}.
\end{proof}

\section{Extension to multi-block problems}
\label{sect 4}
In this section, we focus on using Algorithm \ref{algst}, together with sGS decomposition techniques \cite{li2019block}, to solve multi-block problems.
Let $\cX_{i}$, $i=1,\cdots,p$, $\cY_{j}$,  $j=1,\cdots,q$ and $\cZ$ be finite-dimensional real Hilbert spaces each endowed with inner product $\langle \cdot,\cdot\rangle$ and its induced norm $\|\cdot\|$. Define $\cX:=\mathcal{X}_{1}\times \mathcal{X}_{2}\times \cdots\times \mathcal{X}_{p}$ and $\cY:=\mathcal{Y}_{1}\times \mathcal{Y}_{2}\times \cdots\times \mathcal{Y}_{q}$. 
Let $A: \cX \to\cZ$ and $B:\cY\to\cZ$ be two linear operators, and $c\in\cZ$ be a given vector.
Consider the multi-block optimization problem
\begin{equation}
\label{problem 2}
\begin{array}{cl} 
\min\limits_{x_i\in\cX_i,y_j\in\cY_j}&   f(x_{1})+h(x_{1},x_{2},\cdots,x_{p})+g(y_{1})+
r(y_{1},y_{2},\cdots,y_{q}) \\
\mbox{s.t. }&  Ax+By=c,
\end{array}
\end{equation}
where $x:=(x_{1},x_{2},\cdots,x_{p})$, $y:=(y_{1},y_{2},\cdots,y_{q})$,
$f:\cX_{1}\to(-\infty,\infty ]$ and $g:\cY_{1}\to (-\infty,\infty ]$ are closed proper convex functions, 
and $h: \cX\rightarrow (-\infty,\infty]$ 
and $r: \cY\rightarrow (-\infty,\infty]$ are convex quadratic functions
\begin{equation*}
h(x):=\frac{1}{2}\langle x,\mathcal{P}x\rangle-
\langle b,x\rangle
\quad \mbox{and}\quad 
r(y):=\frac{1}{2}\langle y,\mathcal{Q}y\rangle-\langle d,y\rangle. 
\end{equation*}
Here, $\mathcal{P}:\cX\to \cX$ and $\mathcal{Q}:\cY\to\cY$ are self-adjoint positive semi-definite linear operators. 
Moreover, the linear operators $A$ and $B$ are given by $Ax=A_{1}x_1+\cdots+A_{p}x_p$ and 
$By:=B_{1}y_1+\cdots+B_q y_q$, respectively,  with $A_i:\cX_i\to \cZ$, $1\le i\le p$ and $B_j:\cY_j\to\cZ$, $1\le j\le q$ being linear operators. 

\begin{remark}
In \cite{li2016schur}, a Schur complement based sPADMM was proposed to solve the problem \eqref{problem 2}, and the convergence was guaranteed by establishing its equivalence to a special sPADMM. 
This equivalence was further generalized in \cite{li2019block}
to a block sGS decomposition theorem, so the algorithm in \cite{li2016schur} is usually called the sGS iteration based sPADMM  (sGS-sPADMM). 
This theorem was used in \cite{chen2018unified} to construct a unified framework of inexact majorized indefinite proximal ADMM. 
Motivated by this line of work and the acceleration of sPADMM studied in the previous section, we propose an accelerated variant of sGS-sPADMM.
\end{remark}

For notational convenience, define
\begin{align*}
&x_{\leq i}:=(x_{1},x_{2},...,x_{i}),\ x_{\geq i}:=(x_{i},x_{i+1},...,x_{p}),\quad i=0,1,...,p+1,\\
&y_{\leq j}:=(y_{1},y_{2},...,y_{j}),\ y_{\geq j}:=(y_{j},y_{j+1},...,y_{q}),\quad j=0,1,...,q+1.
\end{align*}
Moreover, we set $x_{0}=x_{\leq 0}=x_{p+1}=x_{\geq p+1}=y_{0}=y_{\leq 0}=y_{q+1}=y_{\geq q+1}=\varnothing$. 
To apply the results in the previous section to the sGS-sPAMM, 
we should specify two linear operators $\mathcal{T}_{f}:\cX\to\cX$ and $\mathcal{T}_{g}:\cY\to\cY$ and define the proximal augmented Lagrangian function  of problem \eqref{problem 2} by
\begin{equation}
\label{palf}
\begin{array}{ll}
\widetilde{\cal L}_{\lambda}((x,y,z),(\hat x,\hat y)):&=f(x_{1})+h(x)+g(y_{1})+r(y)
+\langle z, Ax+By-c\rangle 
\\[1mm]
&\qquad +\frac{\lambda}{2}\|Ax+By-c\|^{2}
+\frac{\lambda}{2}\|x-\hat x\|^2_{\cT_f}
+\frac{\lambda}{2}\|y-\hat y\|^2_{\cT_g}, 
\end{array}
\end{equation}
where $\lambda>0$ is the penalty parameter. 
Moreover, define two linear operators 
\begin{equation}
\label{MNdefine}
\begin{cases}
\mathcal{M}_{\lambda}:=\mathcal{P}+\lambda A^{*}A+\lambda\mathcal{T}_{f},
\\[1.5mm]
\mathcal{N}_{\lambda}:=\mathcal{Q}+\lambda B^{*}B+\lambda\mathcal{T}_{g}.
\end{cases}
\end{equation}
Note that $\mathcal{M}_{\lambda}$ and $\mathcal{N}_{\lambda}$ in \eqref{MNdefine} can be symbolically as 
\begin{align}\label{MN}
\mathcal{M}_{\lambda}:=\left(
\begin{array}{cccc}
 \mathcal{M}_{11} & \mathcal{M}_{12} & \cdots &\mathcal{M}_{1p} \\
 \mathcal{M}^{*}_{12} &\mathcal{M}_{22} &\cdots &\mathcal{M}_{2p}\\
 \vdots& \vdots &\ddots&\vdots\\
 \mathcal{M}_{1p}^{*} &\mathcal{M}_{2p}^{*}&\cdots &\mathcal{M}_{pp}
\end{array}
\right ) \quad \text{and}\quad
\mathcal{N}_{\lambda}:=\left(
\begin{array}{cccc}
 \mathcal{N}_{11} & \mathcal{N}_{12} & \cdots &\mathcal{N}_{1q} \\
 \mathcal{N}^{*}_{12} &\mathcal{N}_{22} &\cdots &\mathcal{N}_{2q}\\
 \vdots& \vdots &\ddots&\vdots\\
 \mathcal{N}_{1q}^{*} &\mathcal{N}_{2q}^{*}&\cdots &\mathcal{N}_{qq}
\end{array}
\right ),
\end{align}
Similarly, we can
also decompose $\mathcal{T}_{f}$ and $\mathcal{T}_{g}$ as \eqref{MN}.
To ensure that the subproblems are well-defined, we assume that
\begin{equation}
\label{pqcondition}
\begin{cases}
\mathcal{P}_{ii}+\lambda A_{i}^{*}A_{i}+\lambda\mathcal{T}_{f_{ii}}\succ 0, &i=1,\cdots,p,\\[1.5mm]
\mathcal{Q}_{jj}+\lambda B_{j}^{*}B_{j}+\lambda\mathcal{T}_{g_{jj}}\succ 0, &j=1,\cdots,q,
\end{cases}
\end{equation}
where $\mathcal{T}_{f_{ii}}:\cX_{i}\to\cX_{i}$, $i=1,\cdots,p$ and $\mathcal{T}_{g_{jj}}:\cY_{j}\to\cY_{j}$, $j=1,\cdots,q$ are the block diagonal parts
of $\mathcal{T}_{f}$ and $\mathcal{T}_{g}$ , respectively. 

Based on \eqref{MN}, we denote the block-diagonal parts of $\mathcal{M}_{\lambda}$ and $\mathcal{N}_{\lambda}$ using $\mathcal{M}_{d}:=\text{Diag}(\mathcal{M}_{11},\cdots,\mathcal{M}_{pp})$ and $\mathcal{N}_{d}:=\text{Diag}(\mathcal{N}_{11},\cdots,\mathcal{N}_{qq})$, respectively. 
In addition, one can denote the symbolically strictly upper triangular part of $\mathcal{M}_{\lambda}$ by $\mathcal{M}_{u}$, and denote the symbolically strictly upper triangular part of $\mathcal{N}_{\lambda}$ by $\mathcal{N}_{u}$.
To solve the problem \eqref{problem 2}, we obtain the following sGS-AsPADMM algorithm (Algorithm \ref{alg4}) and its convergence rate.

\begin{algorithm}
\caption{The sGS-AsPADMM algorithm for solving \eqref{problem 2}}\label{alg4}
\KwIn{
$x^{0}\in\cX$, $y^{0}\in\cY$, $z^{0}\in\cZ$, $\lambda>0$, $\theta_{0}=1$, $\theta^{-1}=1/\tau$, $\tau \in(0,1)$.
Chose self-adjoint positive semi-definite operators $\mathcal{T}_{f}:\cX\to\cX$, $\mathcal{T}_{g}:\cY\to\cY$ for \eqref{palf}.}
\KwOut{ $\{(x^k,y^k,z^k)\}$.} 
\For{$k=0,1, \ldots$,}{
1. $\theta^{k}:=
\frac{\theta^{k-1}}{\theta^{k-1}(1-\tau)+1}$, $ v^{k}:=y^{k}+\frac{\theta^{k}
(1-\theta^{k-1})}{\theta^{k-1}}(y^{k}-y^{k-1});$\\

2. compute for $i = p,\cdots,3, 2,$
\begin{align*}
\overline{x}_{i}^{k}:= \underset{x_{i}\in\cX_{i}}{\arg\min}~&\widetilde{\cal L}_\frac{\lambda}{\theta^{k}}\big(((x_{\leq i-1}^{k},x_{i},\overline{x}^{k}_{\geq i+1}),v^{k},z^{k}),(x^{k},y^{k})\big );
\end{align*}
3. compute for $i = 1,2,\cdots, p,$
\begin{align*} x_{i}^{k+1}:=\underset{x_{i}\in\cX_{i}}{\arg\min}~ &\widetilde{\cal L}_\frac{\lambda}{\theta^{k}}\big(((x_{\leq i-1}^{k+1},x_{i},\overline{x}^{k}_{\geq i+1}),v^{k},z^{k}),(x^{k},y^{k})\big );
\end{align*}
4. compute for $j = q,\cdots,3, 2,$
\begin{align*} \overline{y}_{j}^{k}= \underset{y_{j}\in\cY_{j}}{\arg\min}~& \widetilde{\cal L}_\frac{\lambda}{\theta^{k}}\big ((x^{k+1},(v_{\leq j-1}^{k},y_{j},\overline{y}^{k}_{\geq j+1}),z^{k}),(x^{k},y^{k})\big);
\end{align*}
5. compute for $j = 1,2,\cdots, q,$
\begin{align*}
y_{j}^{k+1}= \underset{y_{j}\in\cY_{j}}{\arg\min}~& \widetilde{\cal L}_\frac{\lambda}{\theta^{k}}\big ((x^{k+1},(y_{\leq j-1}^{k+1},y_{j},\overline{y}^{k}_{\geq j+1}),z^{k}),(x^{k},y^{k})\big );
\end{align*}
6.  compute $z^{k+1}:=z^{k}-\tau \lambda(Ax^{k+1}+By^{k+1}-c)$.}
\end{algorithm}

Under the premise that \eqref{pqcondition} is true, one can define two self-adjoint positive semi-definite linear operators ${\rm sGS}(\mathcal{M}_{\lambda}):\cX\to\cX$ and ${\rm sGS}(\mathcal{N}_{\lambda}):\cY\to\cY$ by 
\begin{equation}
\label{sGsMN}
\begin{cases}
{\rm sGS}(\mathcal{M}_{\lambda}):=\mathcal{M}_{u}\mathcal{M}_{d}^{-1}\mathcal{M}_{u}^{*},
\\[1.5mm]
{\rm sGS}(\mathcal{N}_{\lambda}):=\mathcal{N}_{u}\mathcal{N}_{d}^{-1}\mathcal{N}_{u}^{*}.
\end{cases}
\end{equation}
Using \cite[Theorem 1]{li2019block}, the subproblems of generating the sequence $\{(x^{k},y^{k})\}$ in Algorithm \ref{alg4} can be equivalently written as 
\begin{equation}
\label{sgssub}
\begin{cases}
x^{k+1}=\underset{x\in\cX}{\arg\min}~\widetilde{\cal L}_{\frac{\lambda}{\theta^k}}((x,v^{k},z^{k}),(x^{k},y^{k}))
+\frac{1}{2}\|x-x^{k}\|_{{\rm sGS}(\mathcal{M}_{\lambda/\theta^k})}^{2}
,\\
y^{k+1}=\underset{y\in\cY}{\arg\min}~\widetilde{\cal L}_{\frac{\lambda}{\theta^k}}((x^{k+1},y,z^{k}),(x^{k},y^{k}))
+\frac{1}{2}\|y-y^{k}\|_{{\rm sGS}(\mathcal{N}_{\lambda/\theta^k})}^{2}.
\end{cases}
\end{equation}
To analyze the convergence of Algorithm \ref{alg4}, we should properly estimate the variable proximal terms in the above subproblems. 
With regard to this issue, we define 
\begin{equation}
\label{defxifg}
\begin{cases}
\Xi_{f}:=&2(1-\tau)\lambda (A^{*}A+\cT_{f})_{u}(\mathcal{P}+\lambda A^{*}A+\lambda \cT_{f})_{d}^{-1}\mathcal{P}_{u}^{*}
\\
&+2(1-\tau)(3-\tau)\lambda (A^{*}A+\cT_{f})_{u}(A^{*}A+\cT_{f})_{d}^{-1}(A^{*}A+\cT_{f})_{u}^{*},
\\[1mm]
\Xi_{g}:=&2(1-\tau)\lambda (B^{*}B+\cT_{g})_{u}(\mathcal{Q}+\lambda B^{*}B+\lambda\cT_{g})_{d}^{-1}(\mathcal{Q}+\mathcal{T}_{g})_{u}^{*}
\\
&+2(1-\tau)(3-\tau)\lambda (B^{*}B+\cT_{g})_{u}(B^{*}B+\cT_{g})_{d}^{-1}(B^{*}B+\cT_{g})_{u}^{*}, 
\end{cases}
\end{equation}
Then we have the following result. 

\begin{proposition}
\label{prop-proximal}
Suppose that $\cT_f$ and $\cT_g$ in \eqref{palf} are chosen such that \eqref{pqcondition} holds, $(A^{*}A+\cT_{f})_{d}\succ 0$, and 
$(B^{*}B+\cT_{g})_{d}\succ 0$. 
Then, one has 
$$
\Xi_{f}\succeq {\rm sGS}(\mathcal{M}_{\frac{\lambda}{\theta^{k+1}}})
-{\rm sGS}(\mathcal{M}_{\frac{\lambda}{\theta^{k}}})
\quad  
\mbox{and}
\quad
\Xi_{g} 
\succeq {\rm sGS}(\mathcal{N}_{\frac{\lambda}{\theta^{k+1}}})
-{\rm sGS}(\mathcal{N}_{\frac{\lambda}{\theta^{k}}}), 
$$ 
where $\Xi_f$ and $\Xi_g$ are defined by \eqref{defxifg}. 
\end{proposition}

\begin{proof}
Fix $k\ge 0$. 
Define $\Xi:={\rm sGS}(\mathcal{M}_{\frac{\lambda}{\theta^{k+1}}})
-{\rm sGS}(\mathcal{M}_{\frac{\lambda}{\theta^{k}}})$
and
$$
\begin{array}{ll}
\hat{\Xi}:=(\frac{\lambda}{\theta^{k+1}} -\frac{\lambda}{\theta^{k}}) \mathcal{H}_{u}(\mathcal{P}+\frac{\lambda}{\theta^{k}} \mathcal{H})_{d}^{-1}(2\mathcal{P}+(\frac{\lambda}{\theta^{k+1}} +\frac{\lambda}{\theta^{k}})\mathcal{H})_{u}^{*}.
\end{array}
$$
Moreover, define $\mathcal{H}:=A^{*}A+\cT_{f}$. Form \eqref{thetarelation} one has $\theta^k > \theta^{k+1}$. 
Then by \eqref{sGsMN} and the definition of $\Theta$ one has  
\begin{equation*}
\begin{array}{rll}
\Xi
\preceq& (\mathcal{P}+\frac{\lambda}{\theta^{k+1}} \mathcal{H})_{u}(\mathcal{P}+\frac{\lambda}{\theta^{k}} \mathcal{H})_{d}^{-1}(\mathcal{P}+\frac{\lambda}{\theta^{k+1}} \mathcal{H})_{u}^{*}\\[1.5mm]
&-(\mathcal{P}+\frac{\lambda}{\theta^{k}} \mathcal{H})_{u}(\mathcal{P}+\frac{\lambda}{\theta^{k}} \mathcal{H})_{d}^{-1}(\mathcal{P}+\frac{\lambda}{\theta^{k}} \mathcal{H})_{u}^{*}\\[1.5mm]
=&\hat{\Xi}+(\mathcal{P}+\frac{\lambda}{\theta^{k}} \mathcal{H})_{u}(\mathcal{P}+\frac{\lambda}{\theta^{k}} \mathcal{H})_{d}^{-1}(\mathcal{P}+\frac{\lambda}{\theta^{k+1}} \mathcal{H})_{u}^{*}
\\[1.5mm]
&-(\mathcal{P}+\frac{\lambda}{\theta^{k+1}} \mathcal{H})_{u}(\mathcal{P}+\frac{\lambda}{\theta^{k}} \mathcal{H})_{d}^{-1}(\mathcal{P}+\frac{\lambda}{\theta^{k}} \mathcal{H})_{u}^{*}. 
\end{array}
\end{equation*}
Thus one has for all $x\in \cX$ that $\langle x,(\hat{\Xi}-\Xi)x\rangle \geq  0$.    
Moreover, for any $x\in\cX$, one can see that 
\begin{equation*}
\begin{array}{rll}
&\langle x,\Xi x\rangle\le\langle x,\hat{\Xi} x\rangle
\\[1.5mm]
=& (1-\tau)\lambda\langle x, \mathcal{H}_{u}(\mathcal{P}+\frac{\lambda}{\theta^{k}} \mathcal{H})_{d}^{-1}(2\mathcal{P}+(\frac{\lambda}{\theta^{k+1}} +\frac{\lambda}{\theta^{k}})\mathcal{H})_{u}^{*}  x\rangle\\[1.5mm]
=&2(1-\tau)\lambda\langle x, \mathcal{H}_{u}(\mathcal{P}+\frac{\lambda}{\theta^{k}} \mathcal{H})_{d}^{-1}\mathcal{P}_{u}^{*} x\rangle\\[1.5mm]
&+(1-\tau)(\frac{\lambda}{\theta^{k+1}} +\frac{\lambda}{\theta^{k}})\lambda\langle x,  \mathcal{H}_{u}(\mathcal{P}+\frac{\lambda}{\theta^{k}} \mathcal{H})_{d}^{-1}\mathcal{H}_{u}^{*} x\rangle\\[1.5mm]
\leq 
&
2(1-\tau)\lambda\langle x, \mathcal{H}_{u}(\mathcal{P}+\lambda \mathcal{H})_{d}^{-1}\mathcal{P}_{u}^{*} x\rangle
+ \frac{2(1-\tau)\lambda^{2}}{\theta^{k+1}}\langle x, \mathcal{H}_{u}(\mathcal{P}+\frac{\lambda}{\theta^{k}} \mathcal{H})_{d}^{-1}\mathcal{H}_{u}^{*} x\rangle\\[1.5mm]
\leq & 2(1-\tau)\lambda \langle x,\mathcal{H}_{u}(\mathcal{P}+\lambda \mathcal{H})_{d}^{-1}\mathcal{P}_{u}^{*} x\rangle
+\frac{2(1-\tau)\lambda\theta^{k}}{\theta^{k+1}}\langle x, \mathcal{H}_{u}(\frac{\theta^{k}}{\lambda}\mathcal{P} +\mathcal{H})_{d}^{-1}\mathcal{H}_{u}^{*} x\rangle\\[1.5mm]
\leq &2(1-\tau)\lambda \langle x,\mathcal{H}_{u}(\mathcal{P}+\lambda \mathcal{H})_{d}^{-1}\mathcal{P}_{u}^{*} x\rangle
+2(1-\tau)(3-\tau)\lambda\langle x, \mathcal{H}_{u}\mathcal{H}_{d}^{-1}\mathcal{H}_{u}^{*} x\rangle\\[1.5mm]
=&\langle x,\Xi_{f}x\rangle.

\end{array}
\end{equation*}
Similarly, one can obtain $\Xi_{g} 
\succeq {\rm sGS}(\mathcal{N}_{\frac{\lambda}{\theta^{k+1}}})
-{\rm sGS}(\mathcal{N}_{\frac{\lambda}{\theta^{k}}})$, 
and this completes the proof. 
\qed
\end{proof}

Based on Proposition \ref{prop-proximal}, we have the following convergence result for Algorithm \ref{alg4}.

\begin{theorem}
\label{convergence-multi}
Suppose that the KKT system of problem \eqref{problem 2} admits a solution.  
Suppose that $\cT_f$ and $\cT_g$ in \eqref{palf} are chosen such that \eqref{pqcondition} holds, $(A^{*}A+\cT_{f})_{d}\succ 0$ and $(B^{*}B+\cT_{g})_{d}\succ 0$. 
Assume that $\Xi_f$ and $\Xi_g$ defined in \eqref{defxifg} satisfy $\Sigma_{f}\succeq \Xi_{f}$ and $\Sigma_{g}\succeq \Xi_{g}$.
Then Algorithm \ref{alg4} is well defined, and it generates the sequence
$\{(x^{k},y^{k},z^{k})\}$ such that both
$\{f(x_{1}^{k})+h(x^{k})+g(y_{1}^{k})+r(y^{k})\}$ converge to the optimal value and $\|Ax+By-c\|$ converge to $0$, respectively, with an $O(1/ K)$ non-ergodic convergence rate.  
\end{theorem}

\begin{proof}
Since \eqref{pqcondition} holds, the subproblems of Algorithm \ref{alg4} can be equivalently recast as \eqref{sgssub}. 
Consequently, Algorithm \ref{alg4} can be seen as a special case of Algorithm \ref{algst}. 
Moreover, since the conclusion of Proposition \ref{prop-proximal} holds, from $\Sigma_{f}\succeq \Xi_{f}$ and $\Sigma_{g}\succeq \Xi_{g}$ it can be seen that Theorem \ref{thmst} is sufficient to ensure the $O(1/ K)$ non-ergodic convergence rates for both the objective value and the feasibility. 
This completes the proof. 
\qed
\end{proof}

\section{Applications}
\label{sect 5}
In this section, we study the applications of the algorithms proposed in this paper. 
These applications include robust low-rank tensor completion (RTC) problems \cite{wu2023robust}, $\ell_{0}$ regularization problems for mixed sparse optimization \cite{deng2013group}, and a linear homogeneous equation \cite{chen2016direct}. 
The two non-convex applications have a natural difference-of-convex (DC) structure, so the DC algorithm is preferable to solve them, in which the multi-block convex subproblems can be tackled by the proposed sGS-AsPADMM (Algorithm \ref{alg4}). 
Compared with single-loop algorithms for nonconvex problems such as
\cite{zeng2022moreau,zhang2020proximal,zhang2022global}, the DC algorithm is better suited to the problems we consider and has the potential to converge to points
with better stationary properties \cite{pang2016}.

\subsection{Application to RTC problems}
The original model of RTC problems aims to minimize an objective that consists of the tensor rank function plus the $\ell_0$ norm under limited sample constraints, i.e., 
\begin{equation}
\label{RTC 1}
\begin{array}{cl} 
\min\limits_{\bm{G},\bm{M}} & \text{rank}_{a}(\bm{G})+\lambda\|\bm
{M}\|_{0}\\
\mbox{s.t.} & \Pi_{\Omega}(\bm{G}+\bm{M})=\Pi_{\Omega}(\bm{X}),\ 
\|\bm{G}\| \leq j_{1},\ 
\|\bm{M}\|_{\infty} \leq j_{2},
\end{array}
\end{equation}
where $\bm{G}\in \mathbb{R}^{n_{1}\times n_{2}\times n_{3}}$ represents the restored tensor data and $\bm{M}\in\mathbb{R}^{n_{1}\times n_{2}\times n_{3}}$ represents the tensor data of resulting noise, 
$\text{rank}_{a}(\bm{G})$ is the tensor average rank of $\bm{G}$, 
$\|\cdot \|_{0}$ is the number of nonzero elements, 
$ \|\cdot\|_{\infty}$ and $\|\cdot \|$ are the infinity norm and tensor spectral norm, respectively, $\bm{X}\in \mathbb{R}^{n_{1}\times n_{2}\times n_{3}}$ is the tensor data of complete picture, $j_{1}$ and $j_{2}$ are given constants, 
and $\Omega$ is a set of indices for tensors in $\mathbb{R}^{n_{1}\times n_{2}\times n_{3}}$. 
Furthermore, the projection operator $\Pi_{\Omega}(\cdot)$ is defined by
$$
\Pi_{\Omega}(\bm{X}): = 
\begin{cases}
\bm{X}_{ijs}, &\mbox{if } ijs \mbox{ is an index in } \Omega,
\\
0, &\text{otherwise}.	
\end{cases}
$$
According to \cite{zhao2022robust}, model \eqref{RTC 1} can be equivalently written as the following bound-constrained nonconvex optimization problem
\begin{equation}
\label{RTC 2}
\begin{array}{cl}
\min\limits_{\bm{G},\bm{M}}  &\|\bm{G}\|_{\rm TNN}-H_{1}
(\bm{G})+\lambda(\|\bm{M}\|_{1}-H_{2}(\bm{M}))
\\
\mbox{s.t.} 
& 
\Pi_{\Omega}(\bm{G}+\bm{M})=\Pi_{\Omega}(\bm{X}),
\ 
\|\bm{G}\| \leq j_{1},
\ 
\|\bm{M}\|_{\infty} \leq j_{2},
\end{array}
\end{equation}
where $\|\bm{G}\|_{\rm TNN}$ is the TNN norm of $\bm{G}$ \cite{Lu2019tensor}, $H_{1}$ and
$H_{2}$ are convex function defined by
$$H_{1}(\bm{G}):=\frac{1}{n_{3}}\sum_{i=1}^{n_{3}}g(\sum(\hat{\bm{G}}
^{i})),\quad  
H_{2}(\bm{M}):=\sum_{i=1}^{n_{1}}\sum_{j=1}^{n_{2}}
\sum_{h=1}^{n_{3}}h(\bm{M}_{ijk}),
$$
in which $\hat{\bm{G}}^{i}$ represents the $i$-th frontal slice of the tensor obtained after the fast Fourier transform (FFT) of $\bm{G}$, $g(x):=\sum_{i=1}^{{\rm dim}(x)}h(x_{i})$, and $h$ is a continuously differentiable convex function associated with the MCP and SCAD functions. Specifically, with $\gamma$, $\gamma_1$ and $\gamma_2$ being positive real numbers such that $\gamma_1\le \gamma_2$,  
\begin{equation}\notag
	h(x_{i}) := \begin{cases}
	\frac{x_{i}^{2}}{2\gamma}~, &\lvert x_{i}\rvert \leq \gamma,\\
	\lvert x_{i}\rvert-\frac{\gamma}{2}~, 	&\lvert x_{i}\rvert > \gamma,
		   \end{cases}
\end{equation}
or
\begin{equation}\notag
	h(x_{i}) := \begin{cases}
	0~, &\lvert x_{i}\rvert \leq \gamma_{1},\\
    \frac{x_{i}^{2}-2\gamma_{1}\lvert x_{i}\rvert+\gamma_{1}^{2} }{2(\gamma_{2}-\gamma_{1})},~&\gamma_{1}<\lvert x_{i}\rvert\leq \gamma_{2},\\
	\lvert x_{i}\rvert-\frac{\gamma_{1}+\gamma_{2}}{2}~, 	&\lvert x_{i}\rvert > \gamma_{2}.
		   \end{cases}
\end{equation}
For \eqref{RTC 2}, \cite{zhao2022robust} proposed a proximal majorization-minimization (PMM) algorithm, in which the subproblem at the $(n+1)$-th iteration was given by
\begin{equation}
\label{RTC 3}
\begin{array}{cl}
\min\limits_{\bm{G},\bm{M},\bm{Z}} &\mathcal{F}_{n}
(\bm{G},\bm{M},\bm{Z})\\[1mm]
\mbox{s.t.} &\bm{G}+\bm{M}=\bm{Z},
\end{array}
\end{equation}
where, with $(\bm{G}^n,\bm{M}^n,\bm{Z}^n)$ being the solution of the subproblem at the $n$-th iteration, 
$$
\begin{array}{ll}
\mathcal{F}_{n}
(\bm{G},\bm{M},\bm{Z}):= &
\|\bm{G}\|_{{\rm TNN}}-
\langle \nabla H_{1}(\bm{G}^{n}),\bm{G} \rangle+\lambda(\|
\bm{M}\|_{1}-\langle \nabla H_{2}(\bm{M}^{n}),\bm{M} \rangle)
\\[1mm]
&+\frac{\eta}{2}\|\bm{G}-\bm{G}^{n}\|^{2}
+\frac{\eta}{2}\|\bm{M}-\bm{M}^{n}\|^{2}+\frac{\eta}{2}\|
\bm{Z}-\bm{Z}^{n}\|^{2}
\\[1mm]
&+\delta_{D_{1}}(\bm{G})+\delta_{D_{2}}(\bm{M})+\delta_{D_{3}}(\bm{Z}),
\end{array}
$$
where $\|\cdot\|$ means the Frobenius norm of tensor, 
$\eta>0$, 
$D_{1}:=\{\bm{G} \mid \|\bm{G}\| \leq j_{1}\}$, 
$D_{2}:=\{\bm{M} \mid \|\bm{M}\|_{\infty} \leq j_{2}\}$ and 
$D_{3}:=\{\bm{Z} \mid \Pi_{\Omega}(\bm{Z})=\Pi_
{\Omega}(\bm{X})\}$.

Note that \eqref{RTC 3} is a separable three-block convex problem, which fits the scope of Algorithm \eqref{alg4}. 
Since the augmented Lagrangian function of problem \eqref{RTC 3} is defined by 
\begin{align*}
\mathcal{L}_{\beta}(\bm{Z},\bm{G},\bm{M},\mu):=&\|\bm{G}\|_{{\rm TNN}}-
\langle \nabla H_{1}(\bm{G}^{n}),\bm{G} \rangle+\lambda(\|
\bm{M}\|_{1}-\langle \nabla H_{2}(\bm{M}^{n}),\bm{M} \rangle)\notag\\
&+\frac{\eta}{2}\|\bm{G}-\bm{G}^{n}\|^{2}
+\frac{\eta}{2}\|\bm{M}-\bm{M}^{n}\|^{2}+\frac{\eta}{2}\|
\bm{Z}-\bm{Z}^{n}\|^{2}\\&+\langle \mu,\bm{Z}-\bm{G}-\bm{M}\rangle
+\frac{\beta}{2}\|\bm{Z}-\bm{G}-\bm{M}\|^{2}, 
\end{align*}
the resulting application of Algorithm \ref{alg4} (without adding the proximal terms) can be given as the following algorithm.

\begin{algorithm}
\caption{The sGS-AsPADMM algorithm for solving \eqref{RTC 3}}\label{alg6}
\KwIn{
$\bm{Z}^{0}$, $\bm{G}^{0}$, $\bm{M}^{0}$, $\mu^{0}$, $\beta>0$, $\tau \in (0,1)$, $\theta^{-1}=1/\tau$, $\eta>0$, $\lambda>0$.}
\KwOut{ $\{(\bm{Z}^k,\bm{G}^k,\bm{M}^k)\}$.} 
\For{$k=0,1, \ldots$,}{
~~1. $\theta^{k}:=
\frac{\theta^{k-1}}{\theta^{k-1}(1-\tau)+1}$, $ \bm{V}^{k}=\bm{M}^{k}+\frac{\theta^{k}
(1-\theta^{k-1})}{\theta^{k-1}}(\bm{M}^{k}-\bm{M}^{k-1})$;\\
\begin{align*}
2. ~\bm{Z}^{k+\frac{1}{2}}:&=\underset{\bm{Z}\in D_{3}}
{{\rm \arg\min}}\, \mathcal{L}_{\frac{\beta}{\theta^{k}}}(\bm{Z},\bm{G}^{k},\bm{V}^{k},\mu^{k})\\
&=\Pi_{\Omega}(\bm{X})+\Pi_{\overline{\Omega}}
\{(\eta\theta^{k} I+\beta I)^{-1}(\theta^{k}\eta \bm{Z}^{n}+\beta(\bm{G}^{k}+\bm{V}^{k})-\theta^{k}\mu^{k})\};
\end{align*}
~~3. $\bm{G}^{k+1}:=\underset{\bm{G}}
{{\rm \arg\min}}\, \mathcal{L}_{\frac{\beta}{\theta^{k}}}(\bm{Z}^{k+\frac{1}{2}},\bm{G},\bm{V}^{k},\mu^{k});~$\\
\begin{align*}
4.~\bm{Z}^{k+1}:&=\underset{\bm{Z}\in D_{3}}
{{\rm \arg\min}}\, \mathcal{L}_{\frac{\beta}{\theta^{k}}}(\bm{Z},\bm{G}^{k+1},\bm{V}^{k},\mu^{k})\\
&=\Pi_{\Omega}(\bm{X})+\Pi_{\overline{\Omega}}\{
(\eta\theta^{k} I+\beta I)^{-1}(\theta^{k}\eta \bm{Z}^{n}+\beta(\bm{G}^{k+1}+\bm{V}^{k})-\theta^{k}\mu^{k})\},
\end{align*}
~~5. $ \bm{M}^{k+1}:=\underset{\mathcal{M}}
{{\rm \arg\min}}\, \mathcal{L}_{\frac{\beta}{\theta^{k}}}(\bm{Z}^{k+1},\bm{G}^{k+1},\bm{M},\mu^{k});$\\
~~6. $\mu^{k+1}:=\mu^{k}+\tau\beta(\bm{Z}^{k+1}-\bm{G}^{k+1}
-\bm{M}^{k+1}).$}
\end{algorithm}

\begin{remark}
In Algorithm \ref{alg6}, $\overline{\Omega}$ denotes the indices of the tensors in $\mathbb{R}^{n_{1}\times n_{2}\times n_{3}}$ that are not in $\Omega$.
Since $D_3$ is a subspace, the sGS decomposition theorem is still applicable. 
In addition, the subproblem for computing $\bm{G}^{k+1}$ can be achieved by the tensor singular value decomposition, and the subproblem for getting $\bm{M}^{k+1}$ is just a simple soft shrinkage thresholding operation.  
Moreover, regarding Algorithm \ref{alg6} as Algorithm \ref{alg4} and using the terminology for the latter, we can set $\tau\in(2-\sqrt{1+\frac{\eta}{2\lambda}}, 1)$, so that 
$$
\Sigma_{f}:=\eta I_{\bm{G}+\bm{Z}} \succeq \Xi_{f}:=2(1-\tau)(3-\tau)
\lambda\begin{pmatrix} I_{\bm{G}} & 0\\ 0 & 0 \end{pmatrix}.
$$   

Therefore, Theorem \ref{convergence-multi} can be applied to guarantee the $O(1/ K)$ non-ergodic convergence rates of both the objective and the feasibility of Algorithm \ref{alg6}. 
\end{remark}

In our numerical experiment, we compare the performance of Algorithms \ref{alg1} and Algorithms \ref{alg6} and qualify the approximate solutions by the duality gap and the primal infeasibility, i.e., 
$$
\varepsilon_{\rm gap}:=\frac{\mid {\rm pobj}-{\rm dobj}\mid}{1+\mid {\rm pobj}\mid +\mid {\rm dobj}\mid},
\quad \text{and} \quad\varepsilon_{p}:=\frac{\|\bm{Z}-
\bm{G}-\bm{M}\|_{F}}{1+\|\bm{Z}\|_{F}+\| \bm{G}\|_{F}+
\| \bm{M}\|_{F}} ,
$$
where ``pobj'' and ``dobj'' are the primal and dual objective values, respectively. 
We terminate the algorithms if $\max\{\varepsilon_{\rm gap},\varepsilon_{\rm p}\}$ is less than or equal to a certain threshold $\epsilon$, or the number of iterations is greater than $200$. 
The test problems are constructed from the data including ``peppers and women''\footnote{\url{http://sipi.usc.edu/database/}}, ``flowers and houses''\footnote{\url{https://www2.eecs.berkeley.edu/Research/Projects/CS/vision/bsds/}}. 
Specifically, for a given tensor, we use a certain sample ratio $\text{SR}=\frac{\mid\Omega\mid}{n_{1}\times n_{2}\times n_{3}}$ to construct the corresponding RTC problem, where $|\Omega|$ is the cardinality of $\Omega$. 
For each tensor, we randomly add the salt-and-pepper impulse noise in the ratio $\alpha=0.2$ \cite{chan2005salt}.
For all the problems tested, the penalty parameters ($\gamma$, $\gamma_1$, $\gamma_2$) are selected based on the performance of sGS-sPADMM and then fixed for all other algorithms. 
We set the penalty parameter $\beta=0.1$ in the augmented Lgrangian function and $\tau=0.95$ for Algorithm \ref{alg6}.

Since sGS-AsPADMM only accelerates the process to find the solution to the subproblem \eqref{RTC 3}, we first compare the time and number of iterations taken to solve the subproblem for the first time. Moreover, to show the superiority of the accelerated Algorithm $\ref{alg6}$, we use the ADMM algorithm directly extended to 3-block (ADMM-3d) to solve the subproblem \eqref{RTC 3} for comparison. 
For different images and different disturbing noises, we use ADMM-3d, sGS-sPADMM, and sGS-AsPADMM to solve the subproblems \eqref{RTC 3}, respectively, and the performance data of the different methods are shown in Table \ref{table1}, where iter and time represent the number of iterations and time spent to solve the subproblem \eqref{RTC 3} for the first time, respectively.

\begin{table}
\begin{center}
\begin{tabular}{ |m{0.6cm}<{\centering}|m{1.5cm}<{\centering}|m{1cm}<{\centering}
|m{1cm}<{\centering}|m{1cm}<{\centering}|m{1cm}<{\centering}
|m{1cm}<{\centering}|m{1cm}<{\centering}| }
 \hline
 \multirow{2}{1.3em}{SR} & \multirow{2}{3em}{picture} & \multicolumn{2}{|c|}{ADMM-3d} & \multicolumn{2}{|c|}{sGS-sPADMM} &\multicolumn{2}{|c|}{sGS-AsPADMM}
 \\
 \cline{3-8}
& & iter & time  & iter & time& iter & time  \\
 \hline
 \multirow{4}{2em}{~~0.4}
& pepper & 23 & 12.631  & 21 & 11.961&12 & 6.586  \\
 \cline{2-8}
& flower & 23 & 5.430  & 20 & 4.844 & 10 &2.438 \\
\cline{2-8}
& house & 22 & 2.390 & 16 & 1.766 & 9& 1.099 \\
\cline{2-8}
& women & 24 & 13.215  & 22 & 12.189 & 12 & 6.765 \\
\cline{2-8}
 \hline
 \multirow{4}{2em}{~~0.5}
& pepper & 20 & 10.953  & 18 & 10.286&10 &5.791   \\
 \cline{2-8}
& flower & 21 & 5.044  & 17 & 4.146 &9 &2.207 \\
\cline{2-8}
& house & 21 & 2.297 & 14 & 1.609 & 10& 1.231 \\
\cline{2-8}
& women & 21 & 11.623  & 19 & 10.822 &10 & 5.698 \\
\cline{2-8}
 \hline
  \multirow{4}{2em}{~~0.6}
& pepper & 19 & 10.556  & 15 & 8.307&9 & 4.983  \\
 \cline{2-8}
& flower & 19 & 4.544  & 14 & 3.469 &8 & 1.969\\
\cline{2-8}
& house & 20 & 2.230 & 15 & 1.716 &9 & 1.114 \\
\cline{2-8}
& women & 20 & 11.200  & 17 & 9.477 & 9& 5.290 \\
\cline{2-8}
 \hline
  \multirow{4}{2em}{~~0.7}
& pepper & 17 & 9.485  & 12 & 6.994& 8& 4.626   \\
 \cline{2-8}
& flower & 18 & 4.470  & 11 & 2.725 & 7& 1.768 \\
\cline{2-8}
& house & 18 & 1.996 & 12 & 1.348 & 9 & 1.047 \\
\cline{2-8}
& women & 18 & 10.355  & 14 & 8.328 & 8& 4.682 \\
\cline{2-8}
 \hline
  \multirow{4}{2em}{~~0.8}
& pepper & 16 & 8.831  & 12 & 6.754& 6 &  3.451  \\
 \cline{2-8}
& flower & 17 & 4.214  & 12 & 3.004 & 5& 1.265 \\
\cline{2-8}
& house & 16 & 1.841 & 12 & 1.445 &7 & 1.001 \\
\cline{2-8}
& women & 16 & 9.403  & 12 & 6.872 &6 & 3.371 \\
\cline{2-8}
 \hline
\end{tabular}
\end{center}
\caption{Comparison of ADMM-3d, sGS-sPADMM and sGS-AsPADMM.}\label{table1}
\end{table}

Moreover, we compared the time and number of steps taken by the overall algorithm and the number of steps required for each sub-iteration. All and the specific results are shown in Table \ref{table5}.

According to Table \ref{table1}, sGS-AsPADMM performs much better than sGS-sPADMM and ADMM-3d. 
In addition, we set the noise level $\alpha$ to $0.2$ and the sample ratio $\text{SR}= 0.8$, and pilot the error reduction curves for different images in Figure \ref{figure1}.
It can be seen from the error reduction curves for different image data that, compared with ADMM-3d and sGS-sPADMM, sGS-AsPADMM decreases the relative error faster and there is obvious progress in time saving.

\begin{landscape}
\begin{table}
\begin{tabular}{ |m{1cm}<{\centering}|m{1.3cm}<{\centering}|m{0.7cm}<{\centering}
|m{0.8cm}<{\centering}
|m{2.5cm}<{\centering}|m{0.7cm}<{\centering}|m{0.8cm}<{\centering}
|m{2.8cm}<{\centering}
|m{0.7cm}<{\centering}|m{0.8cm}<{\centering}
|m{2.5cm}<{\centering}| }
 \hline
 \multirow{2}{2em}{SR} & \multirow{2}{3em}{picture} & \multicolumn{3}{|c|}{ADMM-3d} & \multicolumn{3}{|c|}{sGS-sPADMM} &\multicolumn{3}{|c|}{sGS-AsPADMM}
 \\
 \cline{3-11}
& & total iter &time & iter per step & total iter  &time & iter per step & total iter  &time & iter per step  \\
 \hline
 \multirow{3}{3em}{0.4}
& pepper & 96 & 40.01& 23/17/18/19/19  & 92&38.91 & 21/16/17/18/20&68 &26.68 & 12/14/13/14/15  \\
 \cline{2-11}
& flower & 100 & 21.85   & 23//17/18/20/22  & 117 & 20.68  & 20/16/18/19/21/23 & 67 &11.76   & 10/14/13/14/16 \\
\cline{2-11}
& house & 107 & 19.83  & 22/18/20/22/25 & 117 &  17.51 & 16/16/18/19/22/26 & 68
& 6.83  & 9/13/13/15/18 \\
\cline{2-11}
 \hline
 \multirow{3}{3em}{0.5}
& pepper & 68 &  29.03  & 20/15/16/17  & 60 & 24.87   & 18/13/14/15&45 &   18.28 &10/12/11/12   \\
 \cline{2-11}
& flower & 69 &  15.60  & 21/15/16/17  & 78 & 14.01   & 17/13/14/16/18 &44 &7.72    &9/12/11/12 \\
\cline{2-11}
& house & 91 & 7.63   & 21/16/16/18/20 & 74 & 6.42   & 14/12/13/17/18 & 44& 3.87   & 10/11/11/12 \\
\cline{2-11}
 \hline
  \multirow{3}{3em}{0.6}
& pepper & 47 &  26.38  & 19/14/14  & 51 & 22.44   &15/11/12/13 &29 &  14.09  & 9/10/10  \\
 \cline{2-11}
& flower & 62 &  13.69  & 19/14/14/15  & 51 & 13.91   & 14/11/12/14 &39 &  7.10  & 8/10/10/11\\
\cline{2-11}
& house & 83 & 7.06   & 20/14/15/16/18 & 63 & 5.51   & 15/9/11/13/14 &42 & 4.86   & 9/10/10/13 \\
\cline{2-11}
 \hline
  \multirow{4}{3em}{0.7}
& pepper & 42 & 17.30   & 17/13/12  & 31 &  12.55  & 12/9/10& 24& 9.59   & 8/8/8   \\
 \cline{2-11}
& flower & 59 & 10.63   & 18/13/14/14  & 42 &  7.66  & 11/8/10/12 & 35&  6.09  & 7/9/9/10 \\
\cline{2-11}
& house & 59 &  8.83  & 18/13/14/14 & 40 &  5.67  & 12/8/9/11 & 35 &3.21    & 9/8/9/9 \\
\cline{2-11}
 \hline
  \multirow{3}{3em}{0.8}
& pepper & 41 & 16.58   & 16/13/12  & 29 &11.30    & 12/8/9& 22 & 8.81   &  6/8/8  \\
 \cline{2-11}
& flower & 56 & 11.72   & 17/12/13/14  & 39 & 7.03   & 12/8/9/10 & 29&  5.91  & 5/8/9/7 \\
\cline{2-11}
& house & 44 &5.94    & 16/12/7/9 & 35 & 3.31   & 12/6/7/10 &28 &  2.81  & 7/6/8/7 \\
\cline{2-11}
 \hline
\end{tabular}
\caption{Comparison of ADMM-3d, sGS-sPADMM and sGS-AsPADMM.}\label{table5}
\end{table}

\begin{figure} 
\centering
\subfigure[pepper]{
\includegraphics[width=0.375\textwidth]{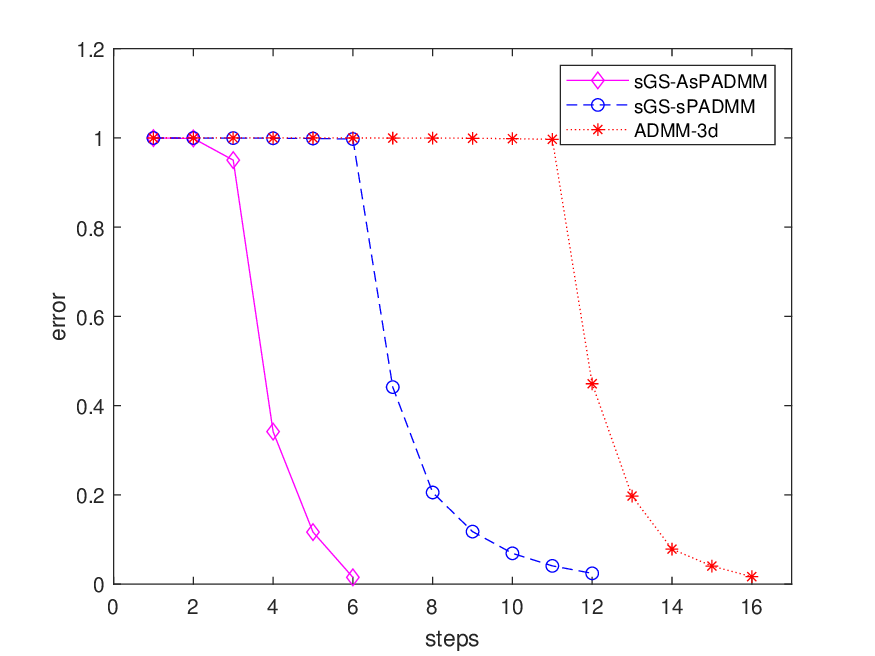} }
\subfigure[flower]{
\includegraphics[width=0.375\textwidth]{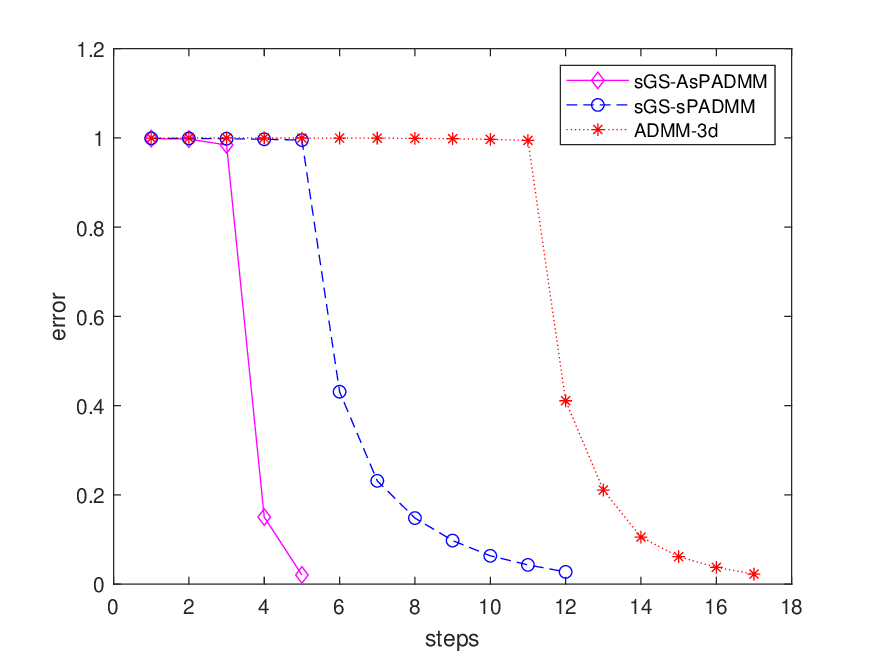} }
\subfigure[woman]{
\includegraphics[width=0.375\textwidth]{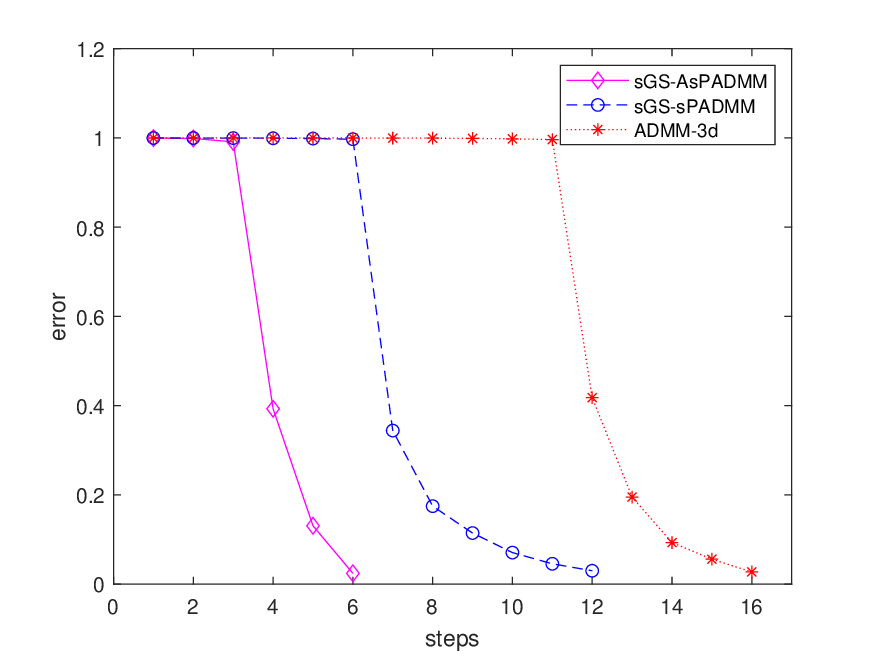}}
\subfigure[house]{
\includegraphics[width=0.375\textwidth]{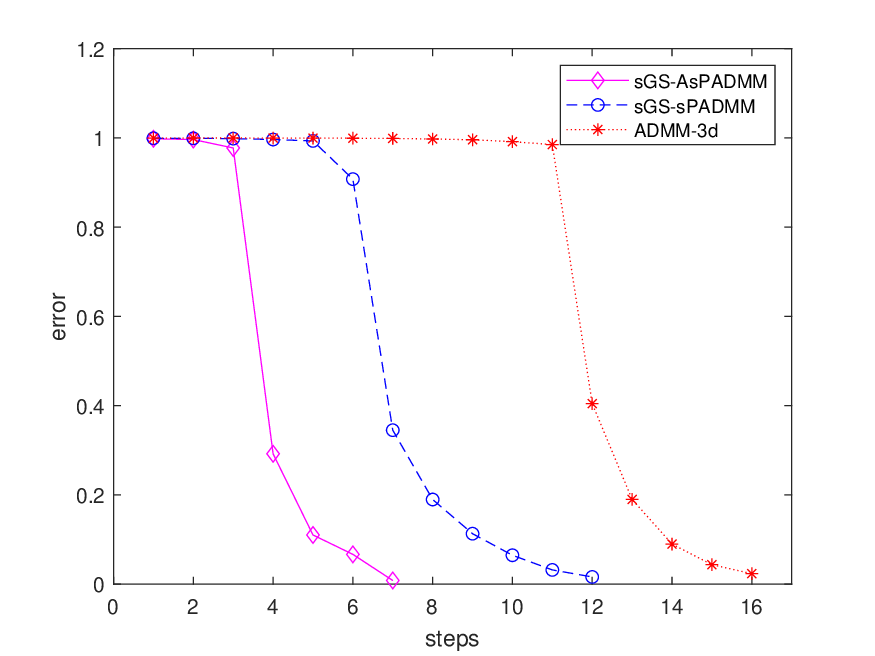}}
\caption{Numerical performance of the  three algorithms for different images}\label{figure1}
\end{figure}

\end{landscape}

\subsection{Application to mix sparse optimization problems}
In this subsection, we consider the following mixed sparse optimization model according to the $\ell_{0}$ and $\ell_{2,0}$ norms:
\begin{align}\label{mix 1}
\underset{x\in \mathbb{R}^{n}}{\min}~\|{\bf A}{\bm x}-b\|^{2}+\lambda_{1}\|{\bm x}\|_{2,0}+
\lambda_{2}\|{\bm x}\|_{0},
\end{align}
where ${\bf A}\in\mathbb{R}^{m\times n}$ is a matrix, $b\in \mathbb{R}^{m}$ is a given vector, $\lambda_{1}$ and $\lambda_{2}$ are two parameters given in advance, 
${\bm x}=({\bm x}_{G_{1}}^{T},{\bm x}_{G_{2}}^{T},\ldots, {\bm x}_{G_{N}}^{T})$ is the target vector to be calculated, 
in which $G_{i}$ denotes the index set of a partial vectors in ${\bm x}$.
Moreover, $\|\cdot\|_{2,0}$ and $\|\cdot\|_{0}$ are defined as $\|{\bm x}\|_{2,0}:=\sum_{i=1}^{N}(\|{\bm x}_{G_{i}}\|_{1})^{0}$ and $ \|{\bm x}\|_{0}:=\sum_{i=1}^{n}
| {\bm x}| ^{0}$, respectively, with the convention that $0^{0}=0$.

Let ${\bm y}_{i}:=\|{\bm x}_{G_{i}}\|_{1}$, then ${\bm y}=({\bm y}_{1},{\bm y}_{2},\ldots,{\bm y}_{N})$. 
Define $\bm{B}\in \mathbb{R}^{N\times n}$ in the sense that the element corresponding to ${\bm x}_{G_{i}}$ in the $i$-th row is $1$, and the rest are zeros.
Then we can convert  \eqref{mix 1} to as follows:
\begin{align}\label{mix 2}
\underset{{\bm x},{\bm y}}{\min}~&\|{\bf A}{\bm x}-b\|^{2}+\lambda_{1}\|{\bm y}\|_{0}+
\lambda_{2}\|{\bm x}\|_{0},\\
\mbox{s.t.} ~&{\bf B}\lvert {\bm x}\rvert ={\bm y}. \notag
\end{align}

By introducing a new variable, we can rewrite $\eqref{mix 2}$ into the following problem,
\begin{equation}
\label{mix 20}
\begin{array}{lll}
\underset{{\bm z},{\bm y}}{\min}~&\|({\bf A},-{\bf A}){\bm z}-b\|^{2}+\lambda_{1}\|{\bm y}\|_{0}+
\lambda_{2}\|{\bm z}\|_{0},\\[1.5mm]
\mbox{s.t.} ~&({\bf B},-{\bf B}){\bm z}={\bm y}, ~{\bm z}\geq 0.
\end{array}
\end{equation}

\begin{proposition}
The problem \eqref{mix 2} is equivalent to the problem \eqref{mix 20}.
\end{proposition}

\begin{proof}
We first prove that any solution to \eqref{mix 20} can be a solution to \eqref{mix 2}. Let $({\bm z}^{*},{\bm y}^{*})$ be a solution of \eqref{mix 20}. Dividing ${\bm z}^{*}$ into $({\bm z}^{*}_{1},{\bm z}^{*}_{2})$, we assert that $({\bm z}^{*}_{1}-{\bm z}^{*}_{2},{\bm y}^{*})$ is a solution to \eqref{mix 2}. Otherwise, there is a solution $(\hat{{\bm x}},\hat{{\bm y}})$ of \eqref{mix 2} such that ${\bf B}|\hat{{\bm x}}|=\hat{{\bm y}}$ and 
$$
\|{\bf A}\hat{{\bm x}}-b\|^{2}+\lambda_{1}\|\hat{{\bm y}}\|_{0}+
\lambda_{2}\|\hat{{\bm x}}\|_{0}<\|{\bf A}({\bm z}_{1}^{*}-{\bm z}_{2}^{*})b\|^{2}+\lambda_{1}\|{\bm y}^{*}\|_{0}+
\lambda_{2}\|({\bm z}_{1}^{*}-{\bm z}_{2}^{*})\|_{0}.
$$
Dividing $\hat{{\bm x}}$ as $\hat{{\bm x}}_{1}-\hat{{\bm x}}_{2}$, where $\hat{{\bm x}}_{1}\geq 0$ and $\hat{{\bm x}}_{2}\geq 0$. Define $\hat{{\bm z}}:=(\hat{{\bm x}}_{1};\hat{{\bm x}}_{2})$, then we have $({\bf B},-{\bf B})\hat{{\bm z}}={\bf B}|\hat{{\bm x}}|=\hat{{\bm y}}$, $\hat{{\bm z}}\geq 0$, and
\begin{equation*}
\begin{array}{lll}
&\|({\bf A},-{\bf A})\hat{{\bm z}}-b\|^{2}+\lambda_{1}\|\hat{{\bm y}}\|_{0}+
\lambda_{2}\|\hat{{\bm z}}\|_{0}=
\|{\bf A}\hat{{\bm x}}-b\|^{2}+\lambda_{1}\|\hat{{\bm y}}\|_{0}+
\lambda_{2}\|\hat{{\bm x}}\|_{0}\\[1.5mm]
<& \|{\bf A}({\bm z}_{1}^{*}-{\bm z}_{2}^{*})b\|^{2}+\lambda_{1}\|{\bm y}^{*}\|_{0}+
\lambda_{2}\|({\bm z}_{1}^{*}-{\bm z}_{2}^{*})\|_{0}\\[1.5mm]
=&
\|({\bf A},-{\bf A}){\bm z}^{*}-b\|^{2}+\lambda_{1}\|{\bm y}^{*}\|_{0}+
\lambda_{2}\|{\bm z}^{*}\|_{0},
\end{array}
\end{equation*}
which contradicts the result that
$({\bm z}^{*},{\bm y}^{*})$ is a solution of \eqref{mix 20}. 

Moreover, we prove that any solution to \eqref{mix 2} can be a solution to \eqref{mix 20}. Let $({\bm x}^{*},{\bm y}^{*})$ be a solution of \eqref{mix 2}. Dividing ${\bm x}^{*}$ into ${\bm x}^{*}={\bm x}_{1}^{*}-{\bm x}_{2}^{*}$, where ${\bm x}^{*}_{1}\geq 0$ and ${\bm x}^{*}_{2}\geq 0$.
We assert that $\big(({\bm x}_{1}^{*};{\bm x}_{2}^{*}),{\bm y}^{*})$ is a solution to \eqref{mix 20}. Or, there is a solution $(\hat{{\bm z}}^{*},\hat{{\bm y}}^{*}):=\big((\hat{{\bm z}}^{*}_{1};\hat{{\bm z}}^{*}_{2}),\hat{{\bm y}}^{*})$ of \eqref{mix 20}, such that $ ({\bf B},-{\bf B})\hat{{\bm z}}^{*}=\hat{{\bm y}}^{*}, ~\hat{{\bm z}}^{*}\geq 0$, and 
\begin{equation*}
\begin{array}{lll}
&\|{\bf A}(\hat{{\bm z}}^{*}_{1}-\hat{{\bm z}}^{*}_{2})-b\|^{2}+\lambda_{1}\|\hat{{\bm y}}^{*}\|_{0}+
\lambda_{2}\|\hat{{\bm z}}^{*}_{1}-\hat{{\bm z}}^{*}_{2}\|_{0}\\[1.5mm]
=&
\|({\bf A},-{\bf A})\hat{{\bm z}}^{*}-b\|^{2}+\lambda_{1}\|\hat{{\bm y}}^{*}\|_{0}+
\lambda_{2}\|\hat{{\bm z}}^{*}\|_{0}\\[1.5mm]
<& \|({\bf A},-{\bf A})({\bm x}_{1}^{*};{\bm x}_{2}^{*})-b\|^{2}+\lambda_{1}\|{\bm y}^{*}\|_{0}+
\lambda_{2}\|({\bm x}_{1}^{*};{\bm x}_{2}^{*})\|_{0}\\[1.5mm]
=&\|{\bf A}{\bm x}^{*}-b\|^{2}+\lambda_{1}\|{\bm y}^{*}\|_{0}+
\lambda_{2}\|{\bm x}^{*}\|_{0}
,
\end{array}
\end{equation*}
which contradicts the result that
$({\bm x}^{*},{\bm y}^{*})$ is a solution of \eqref{mix 2}. 

\end{proof}

Let $\Phi$ be the family of closed proper convex functions
$\phi:\mathbb{R} \to ( - \infty , +\infty]$ that satisfy $[0, 1] \subseteq {\rm int}({\rm dom}\,\phi )$, $\phi (1) = 1$ and $\phi (t^{*}_{\phi}) = 0$, where $t^{*}_{\phi}$
is the unique minimizer of $\phi$ over $[0, 1]$. 
Let $\bm{e}_{n}$ be the $n$-dimensional vector of all ones. 
Then, using the variational properties  of zero norm in the literature  \cite{liu2018equivalent}, the above problem \eqref{mix 20} can be written as
\begin{equation}
\label{mix 3}
\begin{array}{lll}
\underset{{\bm z},{\bm y}}{\min}~&\|({\bf A},-{\bf A}){\bm z}-b\|^{2}+\lambda_{1}\underset{{\bm w}}{\min}\{\sum_
{i=1}^{N}\hat{f}({\bm w}_{i})\}+\lambda_{2}\underset{{\bm v}}{\min}\{\sum_{j=1}^{2n}\hat{g}({\bm v}_{j})
\},\\[1.5mm]
\mbox{s.t.} ~&({\bf B},-{\bf B}){\bm z}={\bm y}, \langle {\bm e}_{N}-{\bm w},| {\bm y}|\rangle=0, 0\leq {\bm w}\leq {\bm e}_{N},\\[1.5mm]
&\langle {\bm e}_{2n}-{\bm v},| {\bm z}|\rangle=0, 0\leq {\bm v}\leq {\bm e}_{2n},{\bm z}\geq 0, 
\end{array}
\end{equation}
where $\hat{f}$, $ \hat{g}\in \Phi$.
For the given $\hat{f}$, $ \hat{g}\in \Phi$, define
\begin{flalign*}
    \quad\quad\quad\quad\quad& \ f_{0}(x):=
    \left\{
        \begin{array}{l}
            \hat{f}(x)~~~x\in [0,1],\\  +\infty\quad~\text{otherwise;}
        \end{array}
    \right.
    & \ &g_{0}(x):=
    \left\{
        \begin{array}{l}
            \hat{g}(x)~~~x\in [0,1],\\  +\infty\quad~\text{otherwise.}
        \end{array}
    \right.
    & &&
\end{flalign*}
By penalizing last two equality constraints in \eqref{mix 3} with two penalty parameters $\rho_{1}>0$ and $\rho_{2}>0$, the resulting problem can be transformed to 
\begin{equation}
\label{mix 4}
\begin{array}{rll}
\underset{{\bm z},{\bm y}}{\min}~&\|({\bf A},-{\bf A}){\bm z}-b\|^{2}+\rho_{1}\|{\bm y}\|_{1}-\lambda_{1}\sum_{i=1}^{N}f_{0}^{*}
(\frac{\rho_{1}}{\lambda_{1}}{\bm y}_{i})\\[1.5mm]
&+\rho_{2}\|{\bm z}\|_{1}-\lambda_{2}\sum_{j=1}^{2n}g_{0}^{*}(
\frac{\rho_{2}}{\lambda_{2}}|{\bm z}_{j}|),\\[1.5mm]
\mbox{s.t.} ~&({\bf B},-{\bf B}){\bm z}={\bm y}, {\bm z}\geq 0, 
\end{array}
\end{equation}
where $f_{0}^{*}$ and $g_{0}^{*}$ are the conjugate functions of $f_{0}$ and $g_{0}$, respectively.
More specifically, we can take $\hat{f}(\chi)=\hat{g}(\chi):=\frac{a^{2}}{4}\chi^{2}-\frac{a^{2}}{2}\chi+a\chi+\frac{(a-2)
^{2}_{+}}{4}$, where $a\geq 2$ and $(a-2)_{+}=\max\{a-2,0\}$. 
Then we can get the expressions of $f_{0}^{*}(\cdot) $ and $g_{0}^{*}(\cdot)$ as follows,
$$
f_{0}^{*}(\breve{\chi})=g_{0}^{*}(\breve{\chi})=
\begin{cases}
 -\frac{(a-2)_{+}^{2}}{4}, &\breve{\chi}\leq \frac{2a-a^{2}}{2},\\
 \frac{1}{a^{2}}(\frac{a^{2}-2a}{2}+\breve{\chi})^{2}-\frac{(a-2)^{2}_{+}}{4}, & \frac{2a-a^{2}}{2}\leq \breve{\chi}<a,\\
 s-1, & \breve{\chi}\geq a.
\end{cases}
$$
Thus we have
$$
|\breve{\chi}|-f_{0}^{*}(|\breve{\chi}|)=|\breve{\chi}|-g_{0}^{*}(|s|)=
\begin{cases}
-\frac{2|\breve{\chi}|}{a}-\frac{\breve{\chi}^{2}}{a^{2}}, & |\breve{\chi}|\leq a,\\
1, &|\breve{\chi}|\geq a.
\end{cases}
$$
When using the PMM algorithm to solve \eqref{mix 4}, the subproblem at the $(t+1)$-th iteration takes the following form 
\begin{equation*}
\begin{array}{rll}
\underset{{\bm z},{\bm y}}{\min}~~\{F({\bm z},{\bm y};{\bm z}^{t},{\bm y}^{t})&:=\|({\bf A},-{\bf A}){\bm z}-b\|^{2}+\rho_{1}\|
{\bm y}\|_{1}+\rho_{2}\|{\bm z}\|_{1}-\lambda_{1}(H_{1}({\bm y}^{t})
\\[1.5mm]
&
+\langle \nabla H_{1}({\bm y}^{t}),{\bm y}-{\bm y}^{t}
\rangle)
+\frac{\eta}{2}\|{\bm y}-{\bm y}^{t}\|^{2}
-\lambda_{2}(H_{2}({\bm z}^{t})
\\[1.5mm]
&
+\langle \nabla H_{2}({\bm z}^{t}),{\bm z}-{\bm z}^{t}\rangle)
+\frac{\eta}{2}\|{\bm z}-{\bm z}^{t}\|^{2}+
\delta_{\Gamma}({\bm z},{\bm y})\},
\end{array}
\end{equation*}
where $\Gamma:=\{(\bm{z},\bm{y})~|~({\bf B},-{\bf B}){\bm z}={\bm y}, {\bm z}\geq 0\}$, $H_{1}({\bm y}):=\sum_{i=1}^{N}f_{0}^{*}(\frac{\rho_{1}}{\lambda_{1}}{\bm y_{i}})$, $H_{2}({\bm z}):=\sum_{j=1}^{2n}g_{0}^{*}(\frac{\rho_{2}}{\lambda_{2}}|{\bm z_{j}}|)$, 
and $({\bm z}^t,{\bm y}^t)$ is the solution of the subproblem at the $t$-th iteration.
The above problem  can be rewritten as 
\begin{equation}
\label{mix 7}
\begin{array}{rll}
\underset{{\bm s},{\bm y},{\bm z}}{\min}~~&\rho_{1}\|
{\bm y}\|_{1}-\lambda_{1}(H_{1}({\bm y}^{t})+\langle \nabla H_{1}({\bm y}^{t}),{\bm y}-{\bm y}^{t}
\rangle)+\frac{\eta}{2}\|{\bm y}-{\bm y}^{t}\|^{2}
+\rho_{2}\|{\bm z}\|_{1}\\[1.5mm]
&-\lambda_{2}(H_{2}({\bm z}^{t})+\langle \nabla H_{2}({\bm z}^{t}),{\bm z}-{\bm z}^{t}\rangle)+\frac{\eta}{2}\|{\bm z}-{\bm z}^{t}\|^{2}+\|({\bf A},-{\bf A}){\bm s}-b\|^{2},\\[1.5mm]
\mbox{s.t.} ~&~({\bf B},-{\bf B}){\bm s}={\bm y}, {\bm z}={\bm s},{\bm s}\geq 0.
\end{array}
\end{equation}
The augmented Lagrangian function of problem \eqref{mix 7} is
\begin{equation*}
\begin{array}{rll}
\mathcal{L}_{\beta}({\bm s},{\bm y},{\bm z},\mu,\hat{\mu}):&= \rho_{1}\|
{\bm y}\|_{1}-\lambda_{1}(H_{1}({\bm y}^{t})+\langle \nabla H_{1}({\bm y}^{t}),{\bm y}-{\bm y}^{t}
\rangle)+\frac{\eta}{2}\|{\bm y}-{\bm y}^{t}\|^{2}
\\[1.5mm]
&
+\rho_{2}\|{\bm z}\|_{1}-\lambda_{2}(H_{2}({\bm z}^{t})+\langle \nabla H_{2}({\bm z}^{t}),{\bm z}-{\bm z}^{t}\rangle)+\frac{\eta}{2}\|{\bm z}-{\bm z}^{t}\|^{2}
\\[1.5mm]
&+\frac{1}{2}\|({\bf A},-{\bf A}){\bm s}-b\|^{2}+\langle \mu,({\bf B},-{\bf B}){\bm s}-{\bm y}\rangle\\[1.5mm]
&+\frac{\beta}{2}\|({\bf B},-{\bf B}){\bm s}-{\bm y}\|^{2}+\langle \hat{\mu},{\bm z}-{\bm s}\rangle+\frac{\beta}{2}\|{\bm z}-{\bm s}\|^{2}+\delta_{\Gamma_{1}}(s)\},
\end{array}
\end{equation*}
where $\Gamma_{1}:=\{{\bm s}~|~{\bm s}\geq 0\}$.

For the above variable separation problem \eqref{mix 7}, 
we can solve it directly using the following sGS-AsPADMM algorithm (without adding the proximal terms). 
\begin{algorithm}
\caption{An sGS-AsPADMM algorithm for solving problem \eqref{mix 7}}\label{alg9}
\KwIn{
${\bm s}^{0},{\bm y}^{0},{\bm z}^{0},\mu^{0},\hat{\mu}^{0},\beta>0,\tau \in (0,1),\theta^{0}=1$.}
\KwOut{$\{({\bm s}^{k+1},{\bm y}^{k+1},{\bm z}^{k+1})\}$}
\For{$k=0,1, \ldots$,}{
1. $\theta^{k}:=
\frac{\theta^{k-1}}{\theta^{k-1}(1-\tau)+1}$, $v^{k}:={\bm y}^{k}+\frac{\theta^{k}
(1-\theta^{k-1})}{\theta^{k-1}}({\bm y}^{k}-{\bm y}^{k-1});$\\
2. $
{\bm s}^{k+\frac{1}{2}}:=\underset{{\bm s}}{{\rm argmin}}\, \mathcal{L}_\frac{\beta}{\theta^{k}}({\bm s},v^{k},{\bm z}^{k},\mu^{k},\hat{\mu}^{k});$\\
3. $
{\bm z}^{k+1}:=\underset{{\bm z}}{{\rm argmin}}\, \mathcal{L}_\frac{\beta}{\theta^{k}}({\bm s}^{k+\frac{1}{2}},v^{k},{\bm z},\mu^{k},\hat{\mu}^{k});$\\
4. $
{\bm s}^{k+1}:=\underset{{\bm s}}{{\rm argmin}}\, \mathcal{L}_\frac{\beta}{\theta^{k}}({\bm s},v^{k},{\bm z}^{k+1},\mu^{k},\hat{\mu}^{k});$\\
5. $
{\bm y}^{k+1}:=\underset{{\bm y}}{{\rm argmin}}\, \mathcal{L}_\frac{\beta}{\theta^{k}}({\bm s}^{k+1},{\bm y},{\bm z}^{k+1},\mu^{k},\hat{\mu}^{k});$\\
6. $\mu^{k+1}:=\mu^{k}+\tau\beta({\bm s}^{k+1}-{\bm z}^{k+1}),~~ \hat{\mu}^{k+1}:=\hat{\mu}^{k}+\tau\beta({\bm y}^{k+1}-({\bf B},-{\bf B}){\bm s}^{k+1});$}
\end{algorithm}

\begin{remark}
In Algorithm \ref{alg9}, the subproblem for computing ${\bm s}^{k+1}$
can be achieved by letting the gradient be zero, and the subproblems for getting ${\bm z}^{k+1}$ and ${\bm y}^{k+1}$ are just a simple soft shrinkage thresholding operation.  
Moreover, regarding Algorithm \ref{alg9} as Algorithm \ref{alg4} and using the terminology for the latter, we can set $\tau\in(\frac{4\lambda+\lambda_{{\bf A}^{*}{\bf A}}-\sqrt{(2\lambda+\lambda_{{\bf A}^{*}{\bf A}})^{2}+2\lambda\eta}}{2\lambda},1]$, so that $\Sigma_{f}:=\eta I \succeq \Xi_{f}$, where $\lambda_{{\bf A}^{*}{\bf A}}$  is the maximum eigenvalue of ${\bf A}^{*}{\bf A}$, and 
$$
\begin{array}{lll}
\Xi_{f}:=&2(1-\tau)\lambda\begin{pmatrix} \big(\lambda I+\lambda({\bf B},-{\bf B})^{*}({\bf B},-{\bf B})+{\bf A}^{*}{\bf A}\big)^{-1}{\bf A}^{*}{\bf A} & 0\\ 0 & 0 \end{pmatrix}\\[1.5mm]
&+2(1-\tau)(3-\tau)
\lambda\begin{pmatrix} \big(I+({\bf B},-{\bf B})^{*}({\bf B},-{\bf B})\big)^{-1} & 0\\ 0 & 0 \end{pmatrix}.
\end{array}   
$$
Therefore, Theorem \ref{convergence-multi} can be applied to guarantee the $O(1/ K)$ non-ergodic convergence rates of both the objective and the feasibility of Algorithm \ref{alg9}. 
\end{remark}

Next, we will compare the numerical performance of sGS-sPADMM and sGS-AsPADMM in the mix sparse optimization problem. We use the same strategy and only compare the performance of two algorithms to solve problems $\eqref{mix 7}$. The dual problem of the problem $\eqref{mix 7}$ is
\begin{equation*}
\begin{array}{lll}
\underset{\mu,\hat{\mu}}{\max}~~&\rho_{1}\|\hat{{\bm y}}\|_{1}-\lambda_{1}\langle \nabla H_{1}({\bm y}^{t}),\hat{{\bm y}}-{\bm y}^{t}\rangle+\frac{\eta}{2}\|\hat{{\bm y}}-{\bm y}^{t}\|^{2}
+\rho_{2}\|\hat{{\bm z}}\|_{1}\\[1.5mm]
&-\lambda_{2}\langle \nabla H_{2}({\bm z}^{t}),\hat{{\bm z}}-{\bm z}^{t}\rangle+\frac{\eta}{2}
\|\hat{{\bm z}}-{\bm z}^{t}\|^{2}+\|({\bf A},-{\bf A})\hat{{\bm s}}-b\|^{2}g\\[1.5mm]
&+\langle \mu,\hat{{\bm s}}-\hat{{\bm z}}\rangle+\langle \hat{\mu},\hat{{\bm y}}-({\bf B},-{\bf B})\hat{{\bm s}}\rangle,
\end{array}
\end{equation*}
where $\hat{{\bm z}}$ and $\hat{{\bm y}}$ are the solutions of $\underset{{\bm z}}{\min}~\{\rho_{2}\|{\bm z}\|_{1}+\frac{\eta}{2}\|{\bm z}-{\bm z}^{t}-\frac
{\mu}{\eta}-\frac{\lambda_{2}}{\eta}\nabla H_{2}({\bm z}^{t})\|^{2}\}$ and $\underset{{\bm y}}{\min}~\{\rho_{1}\|{\bm y}\|_{1}+\frac{\eta}{2}\|{\bm y}-{\bm y}^{t}-\frac
{\hat{\mu}}{\eta}-\frac{\lambda_{1}}{\eta}\nabla H_{1}({\bm y}^{t})\|^{2}\}$, respectively. In addition, $\hat{{\bm s}}$ is the solution of $\underset{{\bm s}\in\Gamma_{1}}{\min}~\|({\bf A},-{\bf A})\hat{{\bm s}}-b\|^{2}+\langle \mu,\hat{{\bm s}}\rangle-\langle \hat{\mu},({\bf B},-{\bf B})\hat{{\bm s}}\rangle$.

In the numerical experiments, the simulation data is generated randomly.  In detail, we randomly generate a matrix ${\bf A}\in \mathbb{R}^{m\times n}$
and a mixed sparse solution ${\bm x}^{*}\in \mathbb{R}^{n}$ by randomly splitting its components into $N$ equisize groups and randomly selecting $S$ of them as non-zero groups, each of which is set as a $r$ sparse vector with a Gaussian ensemble.
Consequently, the inter-group sparsity of the solution is $\frac{S}{N}$, and the intra-group sparsity is $\frac{rN}{n}$. 
The vector obtained is the desired mixed sparse solution. In addition, we generate observational data $b$ as follows,
$$b:={\bf A}\bm{x}^{*}+\varepsilon \times \text{rand}(m,1),$$ where $\varepsilon$ is the standard deviation of additive Gaussian noise, and we set it as $0.1\%$. 

In our numerical experiment, we compare the performance of Algorithms \ref{alg1} and Algorithms \ref{alg9} and qualify the approximate solutions by the duality gap and the primal infeasibility, i.e., 
$$
\varepsilon_{\rm gap}:=\frac{\mid {\rm pobj}-{\rm dobj}\mid}{1+\mid {\rm pobj}\mid +\mid {\rm dobj}\mid},
\quad\varepsilon_{p1}:=\frac{\|\bm{z}-
\bm{s}\|}{1+\|\bm{z}\|+\| \bm{s}\|} ,
$$
and $\varepsilon_{p2}:=\frac{\|(\bf{B,\bf{-B}})\bm{s}-
\bm{y}\|}{1+\|(\bf{B,\bf{-B}})\bm{s}\|+\| \bm{y}\|}$,
where ``pobj'' and ``dobj'' are the primal and dual objective values, respectively. 
We terminate the algorithms if ${\rm error}:=\max\{\varepsilon_{\rm gap},\varepsilon_{p1},\varepsilon_{p2}\}$ is less than or equal
to a certain threshold $\varepsilon:=1e-6$, or the number of iterations is more than 200.

For the various parameters in problem $\eqref{mix 7}$, we set the regularization parameters $(\lambda_{1},\lambda_{2})$ as $(10^{-4},10^{-6})$, and let the penalty parameters $(\rho_{1},\rho_{2}):=(3.3\times10^{-5},3\times10^{-6})$, $\beta=0.05$, $\tau=0.99$. 
In concrete practice, we first use the $\ell_{2,1}$ and $\ell_{1}$ regularization models to perform a low-precision hot start on randomly generated ${\bf A}$ and $b$, and then use the obtained ${\bf A}$ and $b$ as the initial ${\bf A}$ and $b$. 

We compare the error decline curves of the sGS-sPADMM algorithm and the sGS-AsPADMM algorithm, with ${\bf A}$ in different dimensions and solution ${\bm x}$ with different degrees of sparsity as initial conditions. The resulting error decline curves are shown in Figure \ref{figure2}. The two images in the first row are the cases where the sparsity in the group is 0.375, the all-zero group is set to one, and the sparsity within the group is 0.5, the all-zero group is set to two, respectively. In addition, the matrix ${\bf A}$ satisfies ${\bf A}\in\mathbb{R}^{32\times 128}$ in both cases. The two images in the second row are the cases where the sparsity in the group is 0.375, the all-zero group is set to two, and the sparsity within the group is 0.5, the all-zero group is set to four, respectively. Otherwise, the matrix ${\bf A}$ satisfies ${\bf A}\in\mathbb{R}^{64\times 256}$ in both cases.
The results in Figure \ref{figure2} further verify the acceleration of the sGS-AsPADMM algorithm from the sGS-sPADMM algorithm.

\begin{figure}
\centering
\subfigure[case 1]{
\includegraphics[width=0.48\textwidth]{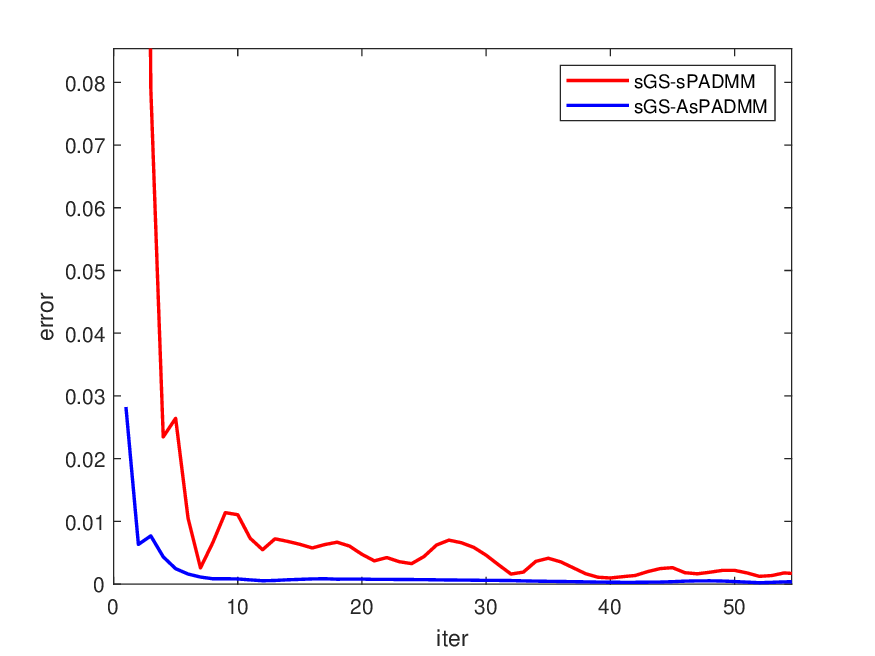}}
\subfigure[case 2]{
\includegraphics[width=0.48\textwidth]{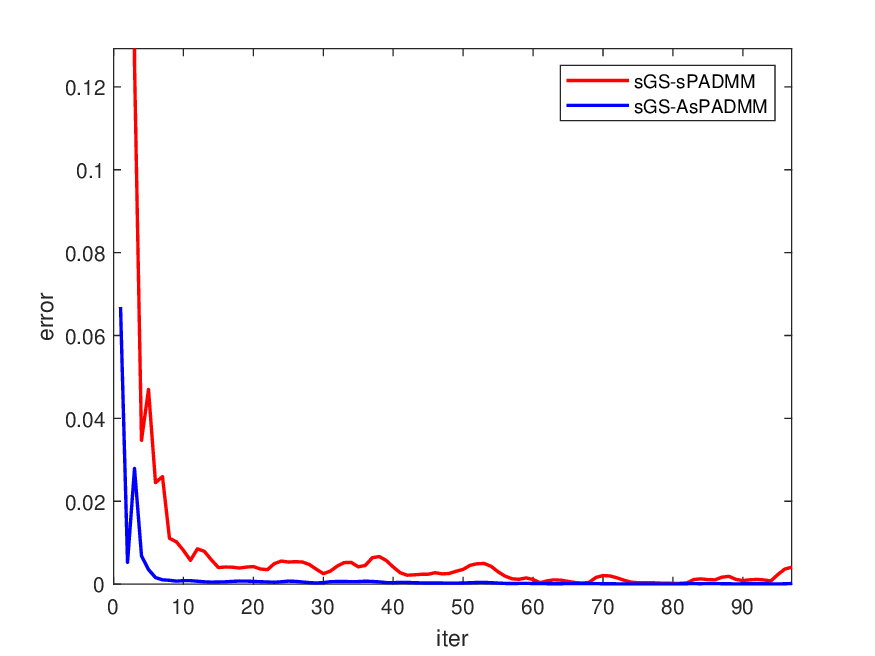} }
\subfigure[case 3]{
\includegraphics[width=0.48\textwidth]{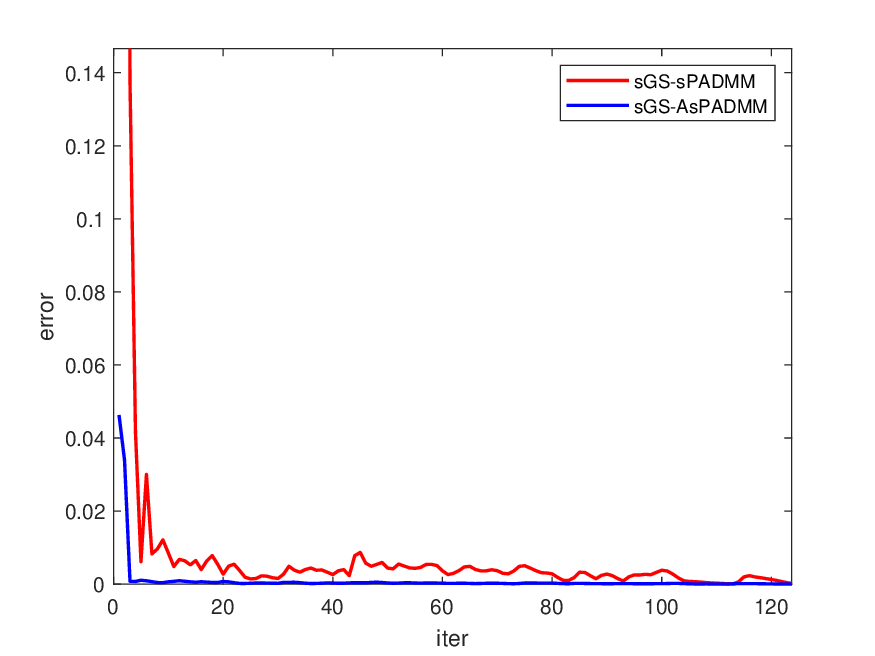}}
\subfigure[case 4]{
\includegraphics[width=0.48\textwidth]{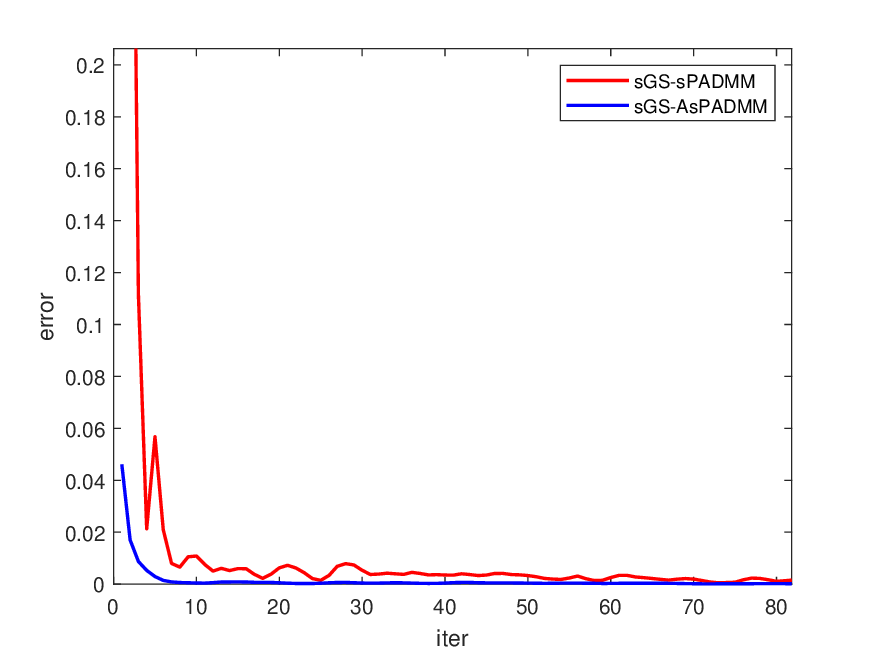}}
\caption{Numerical performance of the two algorithms for different sparsity, where the horizontal axis is the number of iteration steps, and the vertical axis is the iteration error.}\label{figure2}
\end{figure}

\subsection{Application to Lasso problem}
In this subsection, we consider the following Lasso \cite{boyd2011dis} problem 
\begin{equation}
\label{ppro2}
\begin{array}{lll}
\underset{{\bm x}\in\mathbb{R}^{n}}{\min}
~ \frac{1}{2}\|{\bf A}\bm{x}-{\bm b}\|^{2}_{2}+\lambda\|\bm{x}\|_{1},
\end{array}
\end{equation}
 where $\lambda> 0$ is a scalar regularization parameter, $\bf{A}\in\mathbb{R}^{m\times n}$ and $\bm{b}\in\mathbb{R}^m$.
By introducing a new variable $\bm{z}\in\mathbb{R}^n$, problem \eqref{ppro2} becomes
\begin{equation*}
\begin{array}{rll}
\underset{{\bm x},{\bm z}\in\mathbb{R}^{n}}{\min}
&\frac{1}{2}\|{\bf A}\bm{x}-{\bm b}\|^{2}_{2}+\lambda\|\bm{z}\|_{1},\\[1.5mm]
\rm{s.t.}~~&\bm{x}-\bm{z}=0.
\end{array}
\end{equation*}
We use Algorithm \ref{alg2} to solve the above problem \eqref{ppro2}. 
For the proximal terms, we set $\cS:=\lambda_{{\rm max}}I-{\bf A}^{*}{\bf A}$, where $\lambda_{{\rm max}}$ is the maximum eigenvalue of ${\bf A}^{*}{\bf A}$. 
In this way, one has $${\bm x}^{k+1}=\frac{1}{\lambda_{{\rm max}}+\beta}({\bf A}^{*} {\bm b}+\beta \hat{\bm{z}}^{k}+\cS \bm{x}^{k}-{\bm \mu}^{k}).$$ 
Moreover, we set $\cT=0$, so there is an explicit solution for the subproblem about ${\bm z}^{k+1}$.
The resulting algorithm is given as Algorithm \ref{alg10}. 

\begin{algorithm}
\caption{AsPADMM algorithm for solving problem \eqref{ppro2}}
\label{alg10}
\KwIn{
${\bm x}^{0},{\bm z}^{0},{\bm \mu}^{0},\beta>0,\tau \in (0,1),\theta^{0}=1$.}
\KwOut{$\{({\bm x}^{k+1},{\bm z}^{k+1},\mu^{k+1})\}$}
\For{$k=0,1, \ldots$,}{
1. $\theta^{k}=
\frac{\theta^{k-1}}{\theta^{k-1}(1-\tau)+1}$, $\hat{\bm{z}}^{k}=\bm{z}^{k}+\frac{\theta^{k}(1-\theta^{k-1})}{\theta^{k-1}}(\bm{z}^{k}-\bm{z}^{k-1});$\\
2. $
x^{k+1}:=\frac{1}{\lambda_{{\rm max}}+\beta}({\bf A}^{*} {\bm b}+\beta \hat{\bm{z}}^{k}+\cS \bm{x}^{k}-{\bm \mu}^{k});$\\
3. $
\bm{z}^{k+1}:=\underset{z}{\arg\min}~\lambda\|\bm{z}\|_{1}-\langle {\bm \mu}^{k},\bm{z}\rangle+\frac{\beta}{2}\|\bm{z}-\bm{x}^{k+1}\|^{2};$\\
4. $
{\bm \mu}^{k+1}:={\bm \mu}^{k}+\tau\beta(\bm{x}^{k+1}-\bm{z}^{k+1}).$\\}
\end{algorithm}

To terminate Algorithm \ref{alg10}, we take the same strategy as in \cite{boyd2011dis} as follows.
Let ${\bm r}^{k}:=\bm{z}^{k+1}-\bm{x}^{k+1}$
and ${\bm s}^{k}:={\bf A}^{*}({\bf A}{\bm x} ^{k+1}-{\bm b})+\bm{z}^{k+1}$.
We terminate Algorithm \ref{alg10} if $$
\|{\bm r}^{k}\|\leq \epsilon^{\rm{pri}}~~\text{and}~~\|{\bm s}^{k}\|\leq \epsilon^{\rm{dual}},
$$
where $\epsilon^{\rm{pri}}>0$  and $\epsilon^{\rm{dual}}>0$ are feasibility tolerances, which are chosen using an absolute and relative criterion given by
$$
\epsilon^{\rm{pri}}=\sqrt{n}\epsilon^{\rm{abs}}+\epsilon^{\rm{rel}}{\rm max}\{\|{\bm x}^{k+1}\|, \|{\bm z}^{k+1}\|\}~~\text{and}~~\epsilon^{\rm{dual}}=\sqrt{n}\epsilon^{\rm{abs}}+\epsilon^{\rm{rel}}\|{\bm \mu}^{k+1}\|.
$$
Here, $\epsilon^{\rm{abs}}>0$ is an absolute tolerance, $\epsilon^{\rm{rel}}>0$ is a relative tolerance, and ${\bm \mu}$ is the dual variable. 
In our numerical experiments,
we set $\epsilon^{\rm{abs}}=1e-6$ and $\epsilon^{\rm{rel}}=1e-6$. 
We randomly generate $\bf{A}$ with different dimensions and a sparsity solution $\bm{x}^{*}$. 
In addition, we generate observational data $b$ as follows,
$${\bm b}:={\bf A}\bm{x}^{*}+\varepsilon \times \text{rand}(m,1),$$ where $\varepsilon$ is the standard deviation of the additive Gaussian noise, and we set it as $0.1\%$. We compare the performance of the sPADMM and AsPADMM algorithms, and the results are shown in Table  \ref{table 798}.

\begin{table}
\begin{center}
\begin{tabular}{ |m{2cm}<{\centering}|m{2cm}<{\centering}|m{1cm}<{\centering}
|m{1.5cm}<{\centering}|m{2cm}<{\centering}|m{1cm}<{\centering}| }
 \hline
 \multirow{1}{4em}{dim of $A$} & \multirow{1}{5em}{Algorithm} & \multicolumn{1}{|c|}{iter} & \multicolumn{1}{|c|}{time} &\multicolumn{1}{|c|}{objective}
 \\
 \hline
  \multirow{2}{5.5em}{~~$64\times 1028$}
& sPADMM & 1497 & 0.095764  & 0.47 \\
 \cline{2-5}
& AsPADMM & 932 & 0.076896  & 0.47 \\
\cline{2-5}
 \hline
  \multirow{2}{6em}{~~$128\times 1024$}
& sPADMM & 1987 &0.156382  & 0.85  \\
 \cline{2-5}
& AsPADMM & 1310 & 0.109943  & 0.85\\
\cline{2-5}
 \hline
  \multirow{2}{6em}{~~$128\times 2048$}
& sPADMM & 2432 & 0.340028  & 1.19 \\
 \cline{2-5}
& AsPADMM & 1692 & 0.230931  & 1.19 \\
\cline{2-5}
 \hline
\multirow{2}{6em}{~~$256\times 2048$}
& sPADMM & 2591 & 0.621735  & 1.99 \\
 \cline{2-5}
& AsPADMM & 1826 & 0.407821  & 1.99 \\
\cline{2-5}
 \hline
\end{tabular}
\end{center}
\caption{Comparison of sPADMM and AsPADMM in Lasso problem \eqref{ppro2}.}\label{table 798}
\end{table}

As can be seen from the results in Table \ref{table 798}, AsPADMM also has an apparent acceleration effect for pure convex problems.

\section{Conclusion}
\label{sect 6}
In this paper, based on the observation that the sPADMM admits a ${\cal O}(1/\sqrt{K})$ non-ergodic convergence rate, 
we proposed an accelerated sPADMM with an ${\cal O}(1/K)$ 
non-ergodic convergence rate by using extrapolation techniques and monotonically increasing penalty parameters. 
The proposed algorithm, together with the sGS iteration technique, 
can be applied to solve multi-block problems frequently encountered in practice. 
The resulting multi-block algorithm, 
sGS-AsPADMM, also has a ${\cal O}(1/K)$ convergence rate. 
Numerical experiments were conducted by applying the proposed algorithm to robust low-rank tensor completion problems and mixed sparse optimization problems. 
The corresponding numerical results suggested the effectiveness and verified the acceleration phenomena of the proposed sGS-AsPADMM.

\section*{Declaration}

{\bf Conflict of interest} 
The authors have not disclosed any conflict of interest.

\noindent{\bf  Ethical Approval} Not applicable.


\begin{appendix}
\section{Proof of Theorem \ref{thm1}.  }\label{appendix1}
In this Appendix, we give a detailed proof of Theorem \ref{thm1}. 
Let $\{(x^{k},y^{k},z^{k})\}$ be the sequence generated by Algorithm \ref{alg1} with $\tau\in(0,1]$.
For $k\ge 0$, define  
\begin{equation*}
\begin{array}{lll}
 \hat{\nabla} f(x^{k+1}):&=A^{*} z^{k}-\lambda A^{*}(Ax^{k+1}+By^{k}-c)-\mathcal{S}(x^{k+1}-x^{k})\in \partial f(x^{k+1}),\\[1.5mm]
 \hat{\nabla} g(y^{k+1}):&=B^{*} z^{k}-\lambda B^{*}(Ax^{k+1}+By^{k+1}-c)-\mathcal{T}(y^{k+1}-y^{k})\in \partial g(y^{k+1}).
\end{array}
\end{equation*}
Then, for $ \hat{\nabla} f(x^{k+1})$ and $\hat{\nabla} g(y^{k+1})$ defined above,  by converting $z^{k}$ to $z^{k+1}$, we have
\begin{equation}
\label{df x}
\begin{array}{lll}
&~~\quad\langle \hat{\nabla} f(x^{k+1}),x^{k+1}-x\rangle \\[1.5mm]
&=\langle A^{*} z^{k}-\lambda A^{*}(Ax^{k+1}+By^{k}-c)-
\mathcal{S}^{*}(x^{k+1}-x^{k}),x^{k+1}-x\rangle\\[1.5mm]
&=\langle z^{k+1},A(x^{k+1}-x)\rangle+\lambda(\tau-1)\langle Ax^{k+1}+By^{k+1}-c,A(x^{k+1}-x)\rangle\\[1.5mm]
&\quad+\lambda \langle B(y^{k+1}-y^{k}),A(x^{k+1}-x)\rangle-\langle x^{k+1}-x^{k},\mathcal{S}(x^{k+1}-x)\rangle
\end{array}
 \end{equation}
and
\begin{equation}
\label{dg y}
\begin{array}{lll}
&~~\quad\langle \hat{\nabla} g(y^{k+1}),y^{k+1}-y\rangle \\[1.5mm]
&=\langle B^{*} z^{k}-\lambda B^{*}(Ax^{k+1}+By^{k+1}-c)-\mathcal{T}^{*}(y^{k+1}-y^{k}),
y^{k+1}-y\rangle\\[1.5mm]
&=\langle z^{k+1},B(y^{k+1}-y)\rangle+\lambda(\tau-1)\langle Ax^{k+1}+By^{k+1}-c,B(y^{k+1}-y)\rangle\\[1.5mm]
&\quad-\langle y^{k+1}-y^{k},\mathcal{T}(y^{k+1}-y)\rangle.
\end{array}
\end{equation}
Moreover, the following result holds. 
\begin{lemma}
For sequence $\{(x^{k},y^{k},z^{k})\}$ generated by the Algorithm \ref{alg1}, we have
\begin{equation}
\label{dandiao}
\begin{array}{lll}
&\quad~~\frac{1}{\tau\lambda}\|z^{k+1}-z^{k}\|^{2}
+\lambda \|B(y^{k+1}-y^{k})\|^{2}
+\|x^{k+1}-x^{k}\|^{2}_{\mathcal{S}}+\|y^{k+1}-y^{k}
\|^{2}_{\mathcal{T}}\\[1.5mm]
&\leq
\frac{1}{\tau\lambda}\|z^{k}-z^{k-1}\|^{2}
+\lambda\|B(y^{k}-y^{k-1})\|^{2}
+\|x^{k}-x^{k-1}\|^{2}_{\mathcal{S}}+
\|y^{k}-y^{k-1}\|^{2}_{\mathcal{T}}.
\end{array}
\end{equation}
\end{lemma}
\begin{proof}
By utilizing the convexity of ~$f$~ and ~$g$, together with  \eqref{df x} and \eqref{dg y}, we can derive the following inequalities,
\begin{equation}
\label{ff x}
\begin{array}{lll}
0&\leq\langle \hat{\nabla} f(x^{k+1})-\hat{\nabla} f(x^{k}),x^{k+1}-x^{k}\rangle\\[1.5mm]
&=\langle z^{k+1}-z^{k},A(x^{k+1}-x^{k})\rangle+
\lambda(\tau-1)\langle Ax^{k+1}+By^{k+1}-Ax^{k}-By^{k},A(x^{k+1}-x^{k})
\rangle\\[1.5mm]
&\quad+\lambda\langle B(y^{k+1}-2y^{k}+y^{k-1}),A(x^{k+1}-x^{k})\rangle
-\langle x^{k+1}-2x^{k}+x^{k-1},\mathcal{S}
(x^{k+1}-x^{k})\rangle,
\end{array}
\end{equation}
\begin{equation}
\label{gg y}
\begin{array}{lll}
0&\leq\langle \hat{\nabla} g(y^{k+1})-\hat{\nabla} g(y^{k}),y^{k+1}-y^{k}\rangle\\[1.5mm]
&=\langle z^{k+1}-z^{k},B(y^{k+1}-y^{k})\rangle+\lambda(\tau-1)\langle Ax^{k+1}+By^{k+1}-Ax^{k}-By^{k},B(y^{k+1}-y)\rangle\\[1.5mm]
&\quad-\langle y^{k+1}-2y^{k}+y^{k-1},\mathcal{T}(y^{k+1}-y^{k})\rangle.
\end{array}
\end{equation}
From the iterative format of sPADMM, the following equation holds,
\begin{equation}
\label{xtoy}
\begin{array}{lll}
Ax^{k+1}+By^{k+1}-Ax^{k}-By^{k}=\frac{1}{\tau\lambda}(2z^{k}-z^{k+1}-z^{k-1}).
\end{array}
\end{equation}
Combining \eqref{ff x} and \eqref{gg y}, there is
\begin{equation}
\label{suo1}
\begin{array}{lll}
0&\leq\langle \hat{\nabla} f(x^{k+1})-\hat{\nabla} f(x^{k}),x^{k+1}-x^{k}\rangle+\langle \hat{\nabla} g(y^{k+1})-\hat{\nabla} g(y^{k}),y^{k+1}-y^{k}\rangle\\[1.5mm]
&=\langle 
z^{k+1}-z^{k},Ax^{k+1}+By^{k+1}-Ax^{k}-By^{k}\rangle+\lambda(\tau-1)
\|Ax^{k+1}+By^{k+1}-Ax^{k}\\[1.5mm]
&\quad-By^{k}\|^{2}+\lambda\langle B(y^{k+1}-2y^{k}+y^{k-1}),Ax^{k+1}-Ax^{k}\rangle\\[1.5mm]
&\quad-\langle x^{k+1}-2x^{k}+x^{k-1},\mathcal{S}(x^{k+1}-x^{k})
\rangle-\langle y^{k+1}-2y^{k}+y^{k-1},\mathcal{T}(y^{k+1}-y^{k})\rangle\\[1.5mm]
&=-\frac{1}{\tau\lambda}\langle z^{k+1}-z^{k},z^{k+1}-2z^{k}+z^{k-1}\rangle+\frac{\tau-1}
{\tau^{2}\lambda}\|z^{k+1}-2z^{k}+z^{k-1}\|^{2}\\[1.5mm]
&\quad+\lambda\langle B(y^{k+1}-2y^{k}+y^{k-1}),\frac{1}{\tau\lambda}
(2z^{k}-z^{k+1}-z^{k-1})-B(y^{k+1}-y^{k})\rangle\\[1.5mm]
&\quad
-\langle x^{k+1}-2x^{k}+x^{k-1},\mathcal{S}(x^{k+1}-x^{k})
\rangle-\langle y^{k+1}-2y^{k}+y^{k-1},\mathcal{T}(y^{k+1}-y^{k})\rangle\\[1.5mm]
&=\frac{1}{2\tau\lambda}(\|z^{k}-z^{k-1}\|^{2}-\|z^{k+1}-z^{k}\|^{2}
-\|z^{k+1}-2z^{k}+z^{k-1}\|^{2})
\\[1.5mm]
&\quad+\frac{\tau-1}
{\tau^{2}\lambda}\|z^{k+1}-2z^{k}+z^{k-1}\|^{2} -\frac{1}{\tau}\langle B(y^{k+1}-2y^{k}+y^{k-1}),z^{k+1}-2z^{k}+z^{k-1}\rangle \\[1.5mm]
&\quad+\frac{\lambda}{2}(\|B(y^{k}-y^{k-1})\|^{2}-\|B(y^{k+1}-y^{k})\|^{2}
-\|B(y^{k+1}-2y^{k}+y^{k-1})\|^{2})\\[1.5mm]
&\quad+\frac{1}{2}
(\|x^{k}-x^{k-1}\|^{2}_{\mathcal{S}}-\|x^{k+1}-x^{k}
\|^{2}_{\mathcal{S}}-\|x^{k+1}-2x^{k}+x^{k-1}\|^{2}_{\mathcal{S}})\\[1.5mm]
&\quad+\frac{1}{2}(
\|y^{k}-y^{k-1}\|^{2}_{\mathcal{T}}-\|y^{k+1}-y^{k}
\|^{2}_{\mathcal{T}}-\|y^{k+1}-2y^{k}+y^{k-1}\|^{2}_{\mathcal{T}})\\[1.5mm]
&\leq \frac{1}{2\tau\lambda}(\|z^{k}-z^{k-1}\|^{2}-\|z^{k+1}-z^{k}\|^{2})
+\frac{\lambda}{2}(\|B(y^{k}-y^{k-1})\|^{2}-\|B(y^{k+1}-y^{k})\|^{2})\\[1.5mm]
&\quad +\frac{1}{2}(\|x^{k}-x^{k-1}\|^{2}_{\mathcal{S}}-\|x^{k+1}-x^{k}
\|^{2}_{\mathcal{S}})+\frac{1}{2}(
\|y^{k}-y^{k-1}\|^{2}_{\mathcal{T}}-\|y^{k+1}-y^{k}
\|^{2}_{\mathcal{T}}),
\end{array}
\end{equation}
where the second equation converts $z^{k}$ to $z^{k+1}$ and uses the equation \eqref{xtoy}. 
For a detailed description of the last inequality above, we let $a^{k}:=z^{k+1}-2z^{k}+z^{k-1}$ and $b^{k}:= B(y^{k+1}-2y^{k}+y^{k-1})$. Thus for $\tau\in(0,1]$, it holds
\begin{equation*}
\begin{array}{lll}
&~-\frac{1}{2\tau\lambda}\|a^{k}\|^{2}+\frac{\tau-1}
{\tau^{2}\lambda}\|a^{k}\|^{2}-\frac{1}{\tau}\langle a^{k},b^{k}\rangle-\frac{\lambda}{2}\|b^{k}\|^{2}\notag\\[1.5mm]
= &~~\frac{\tau-2}{2\tau^{2}\lambda}\|a^{k}-\frac{\tau\lambda}{\tau-2}
b^{k}\|^{2}-\frac{\lambda(\tau-1)}{2(\tau-2)}\|b^{k}\|^{2}\leq ~0.
\end{array}
\end{equation*}
According to \eqref{suo1}, the lemma is proved.
\qed
\end{proof}
Now we are ready to present the proof of Theorem \ref{thm1}.
\begin{proof}[Proof of Theorem \ref{thm1}]
We separate our proof to the following two parts. 
\begin{itemize}
\item [1)]
When $\tau=1$, combining \eqref{df x} and \eqref{dg y}, 
we can obtain 
\begin{equation}
\label{yuan1}
\begin{array}{lll}
0& \leq f(x^{k+1})+g(y^{k+1})-f(x^{*})-g(y^{*})+\langle z^{*},c-Ax^{k+1}-By^{k+1}\rangle\\[1.5mm]
&\leq \langle \hat{\nabla} f(x^{k+1}),x^{k+1}-x^{*}\rangle +\langle\hat{\nabla} g(y^{k+1}), y^{k+1}-y^{*}\rangle +\langle z^{*},c-Ax^{k+1}-By^{k+1}\rangle\\[1.5mm]
&=-\frac{1}{\lambda}\langle z^{k+1}-z^{*},z^{k+1}-z^{k}\rangle +\langle B(y^{k+1}-y^{k}),z^{k}-z^{k+1}\rangle \\[1.5mm]
&\quad+\lambda \langle B(y^{k+1}- y^{k}), B(y^{*}-y^{k+1})\rangle-\langle y^{k+1}-y^{k},\mathcal{T}(y^{k+1}-y^{*})\rangle\\[1.5mm]
&\quad-\langle x^{k+1}-x^{k},\mathcal{S}(x^{k+1}-x^{*})\rangle\\[1.5mm]
&\leq -\frac{1}{\lambda}\langle z^{k+1}-z^{*},z^{k+1}-z^{k}\rangle -\langle y^{k+1}-2y^{k}+y^{k-1},\mathcal{T}(y^{k+1}-y^{k})\rangle \\[1.5mm]
&\quad+\lambda \langle B(y^{k+1}-y^{k}),~B(y^{*}-y^{k+1})\rangle-\langle y^{k+1}-y^{k},\mathcal{T}(y^{k+1}-y^{*})\rangle\\[1.5mm]
&\quad-\langle x^{k+1}-x^{k},\mathcal{S}(x^{k+1}-x^{*})\rangle\\[1.5mm]
&\leq \frac{1}{2\lambda}(\|z^{k}-z^{*}\|^{2}-\|z^{k+1}-z^{*}\|^{2}-
\|z^{k}-z^{k+1}\|^{2})
\\[1.5mm]
&\quad+\frac{\lambda}{2}(\|B(y^{k}-y^{*})\|^{2}-
\|B(y^{k+1}-y^{*})\|^{2}-\|B(y^{k+1}-y^{k})\|^{2})
\\[1.5mm]
&\quad+\frac{1}{2}(\|x^{k}-x^{*}\|_{\mathcal{S}}
^{2}-\|x^{k+1}-x^{*}\|_{\mathcal{S}}^{2}-\|x^{k+1}-x^{k}\|_{\mathcal{S}}^{2})
\\[1.5mm]
&\quad+\frac{1}{2}(\|y^{k}-y^{*}\|_{\mathcal{T}}^{2}-\|y^{k+1}-y^{*}\|_{\mathcal{T}}^{2}-\|y^{k}-y^{k+1}\|_{\mathcal{T}}^{2}+\|y^{k-1}-y^{k}\|_{\mathcal{T}}^{2}\\[1.5mm]
&\quad-\|y^{k}-y^{k+1}\|_{\mathcal{T}}^{2}-\|y^{k+1}-2y^{k}+y^{k+1}\|^{2}_{\mathcal{T}}),
\end{array}
\end{equation}
where the third inequality uses \eqref{gg y}.
Consequently, it holds that
\begin{equation}
\label{sum1}
\begin{array}{lll}
&\frac{1}{\lambda}\|z^{k}-z^{k+1}\|^{2}+\lambda\|B(y^{k+1}-y^{k})\|^{2}+
\|x^{k+1}-x^{k}\|_{\mathcal{S}}^{2}+\|y^{k+1}-y^{k}\|_{\mathcal{T}}^{2}\\[1.5mm]
\leq &\frac{1}{\lambda}(\|z^{k}-z^{*}\|^{2}-\|z^{k+1}-z^{*}\|^{2})+
\lambda(\|B(y^{k}-y^{*})\|^{2}-\|B(y^{k+1}-y^{*})\|^{2})\\[1.5mm]
\quad&+(\|x^{k}-x^{*}\|_{\mathcal{S}}^{2}-\|x^{k+1}-x^{*}\|_{\mathcal{S}}^{2})+
(\|y^{k}-y^{*}\|_{\mathcal{T}}^{2}-\|y^{k+1}-y^{*}\|_{\mathcal{T}}^{2}
\\[1.5mm]
\quad&+\|y^{k-1}-y^{k}\|^{2}_{\mathcal{T}}-\|y^{k+1}-y^{k}\|_{\mathcal{T}}^{2}).
\end{array}
\end{equation}
Summing \eqref{sum1} from $k=1$ to $K$ and using\eqref{dandiao}, one can get
\begin{equation}
\label{suo11}
\begin{array}{lll}
& K(\frac{1}{\lambda}\|z^{K}-z^{K+1}\|^{2}+\lambda\|B(y^{K+1}
-y^{K})\|^{2}+\|x^{K+1}-x^{K}\|_{\mathcal{S}}^{2}+
\|y^{K+1}-y^{K}\|_{\mathcal{T}}^{2})\\[1.5mm]
\leq &\sum_{k=1}^{K}(\frac{1}{\lambda}\|z^{k}-z^{k+1}\|^{2}+
\lambda\|B(y^{k+1}-y^{k})\|^{2}+\|x^{k+1}-x^{k}\|_{\mathcal{S}}^{2}+
\|y^{k+1}-y^{k}\|_{\mathcal{T}}^{2})\\[1.5mm]
\leq &
C_{1}:=\frac{1}{\lambda}\|z^{1}-z^{*}\|^{2}+\lambda\|B(y^{1}-y^{*})\|^{2}+
\|x^{1}-x^{*}\|_{\mathcal{S}}^{2}+\|y^{1}-y^{*}\|_{\mathcal{T}}^{2}+
\|y^{0}-y^{1}\|_{\mathcal{T}}^{2}
\end{array}
\end{equation}
\item [2)]
When $\tau\in(0,1)$, we let $\hat{a}^{k}=z^{k+1}-z^{k},~\hat{b}^{k}=B(y^{k+1}-y^{k}),~t=\frac
{\tau-1-\sqrt{\tau^{2}-\tau+1}}{2\tau^{2}\lambda}$ and $m=2\tau\lambda t-\frac{\tau-2}{\tau}$. Then, one can get
\begin{equation*}
\begin{array}{lll}
&0<m<1,\quad \frac{\tau-2}{2\tau^{2}\lambda}<t<0,\quad 0<\frac
{1}{4t\tau^{2}}+\frac{\lambda}{2}<\frac{\lambda}{2},\\[1.5mm]
t-\frac{\tau-2}{2\tau^{2}\lambda}&<\frac{1}{2\tau\lambda} \quad\text{~and~}\quad (t-\frac{\tau-2}{2\tau\lambda})/(\frac
{1}{2\tau\lambda})=(\frac
{1}{4t\tau^{2}}+\frac{\lambda}{2})/(\frac{\lambda}{2})=m.
\end{array}
\end{equation*}
It is easy to verify that 
\begin{equation}
\label{suo4}
\begin{array}{lll}
&-\frac{1}{2\tau\lambda}\|\hat{a}^{k}\|^{2}-\frac{\lambda}
{2}\|\hat{b}^{k}\|^{2}-\frac{1}{\tau}\langle\hat{a}^{k},\hat{b}^{k}
\rangle+\frac{\tau-1}{\tau^{2}\lambda}\|\hat{a}^{k}\|^{2}\\[1.5mm]
= & t\|\hat{a}^{k}-\frac{1}{2t\tau}\hat{b}^{k}\|^{2}-
(t-\frac{\tau-2}{2\tau\lambda})\|\hat{a}^{k}\|^{2}-(
\frac{1}{4t\tau^{2}}+\frac{\lambda}{2})\|\hat{b}^{k}\|^{2}.
\end{array}
\end{equation}
Since $f(\cdot)$ and $g(\cdot)$ are convex functions,  using \eqref{df x} and \eqref{dg y} we have 
\begin{equation}
\label{tutu}
\begin{array}{lll}
0& \leq f(x^{k+1})+g(y^{k+1})-f(x^{*})-g(y^{*})+\langle z^{*},c-Ax^{k+1}-By^{k+1}\rangle\\[1.5mm]
&\leq \langle \hat{\nabla} f(x^{k+1}),x^{k+1}-x^{*}\rangle +\langle\hat{\nabla} g(y^{k+1}), y^{k+1}-y^{*}\rangle +\langle z^{*},c-Ax^{k+1}-By^{k+1}\rangle\\[1.5mm]
&=-\frac{1}{\tau\lambda}\langle z^{k+1}-z^{*},z^{k+1}-z^{k}\rangle -\frac{1}{\tau}\langle \hat{a}^{k},\hat{b}^{k}\rangle +\lambda \langle B(y^{k+1}- y^{k}),B(y^{*}-y^{k+1})\rangle\\[1.5mm]
&\quad -\langle y^{k+1}-y^{k},\mathcal{T}(y^{k+1}-y^{*})\rangle-\langle x^{k+1}-x^{k},\mathcal{S}(x^{k+1}-x^{*})\rangle
+\frac{\tau-1}{\tau^{2}\lambda}\|\hat{a}^{k}\|^{2}.
\end{array}
\end{equation}

Combining \eqref{tutu} and \eqref{suo4}, one can see that 
\begin{equation}
\label{suo3}
\begin{array}{lll}
0& \leq f(x^{k+1})+g(y^{k+1})-f(x^{*})-g(y^{*})+\langle z^{*},c-Ax^{k+1}-By^{k+1}\rangle\\[1.5mm]
&=-\frac{1}{\tau\lambda}\langle z^{k+1}-z^{*},z^{k+1}-z^{k}\rangle -\frac{1}{\tau}\langle \hat{a}^{k},\hat{b}^{k}\rangle +\lambda \langle B(y^{k+1}- y^{k}),B(y^{*}-y^{k+1})\rangle\\[1.5mm]
&\quad -\langle y^{k+1}-y^{k},\mathcal{T}(y^{k+1}-y^{*})\rangle-\langle x^{k+1}-x^{k},\mathcal{S}(x^{k+1}-x^{*})\rangle
+\frac{\tau-1}{\tau^{2}\lambda}\|\hat{a}^{k}\|^{2}\\[1.5mm]
&= \frac{1}{2\tau\lambda}(\|z^{k}-z^{*}\|^{2}-\|z^{k+1}-
z^{*}\|^{2}-\|\hat{a}^{k}\|^{2})+\frac{\lambda}{2}(\|B(y^{k}-
y^{*})\|^{2}-\|B(y^{k+1}-y^{*})\|^{2}\\[1.5mm]
&\quad-\|\hat{b}^{k}\|^{2})+\frac{1}{2}(\|x^{k}-x^{*}\|_{\mathcal{S}}^{2}-\|x^{k+1}-x^{*}\|_{\mathcal{S}}^{2}-\|x^{k+1}-x^{k}\|
_{\mathcal{S}}^{2})+\frac{1}{2}(\|y^{k}-y^{*}\|_{\mathcal{T}}^{2}\\[1.5mm]
&\quad-\|y^{k+1}-y^{*}\|_{\mathcal{T}}^{2}-
\|y^{k}-y^{k+1}\|_{\mathcal{T}}^{2})
 -\frac{1}{\tau}\langle\hat{a}^{k},\hat{b}^{k}
\rangle+\frac{\tau-1}{\tau^{2}\lambda}\|\hat{a}^{k}\|^{2}\\[1.5mm]
&\leq  \frac{1}{2\tau\lambda}(\|z^{k}-z^{*}\|^{2}-\|z^{k+1}-z^{*}\|^{2})+\frac{\lambda}{2}(\|B(y^{k}-y^{*})\|^{2}-\|B(y^{k+1}-y^{*})\|^{2})\\[1.5mm]
&\quad+\frac{1}{2}(\|x^{k}-x^{*}\|_{\mathcal{S}}^{2}-\|x^{k+1}-x^{*}\|_{\mathcal{S}}^{2}-\|x^{k+1}-x^{k}\|_{\mathcal{S}}^{2})
-(t-\frac{\tau-2}{2\tau\lambda})\|\hat{a}^{k}\|^{2}
\\[1.5mm]
&\quad+\frac{1}{2}(\|y^{k}-y^{*}\|_{\mathcal{T}}^{2}
-\|y^{k+1}-y^{*}\|_{\mathcal{T}}^{2}-
\|y^{k}-y^{k+1}\|_{\mathcal{T}}^{2})
-(\frac{1}{4t\tau^{2}}+\frac{\lambda}{2})\|\hat{b}^{k}\|^{2}.
\end{array}
\end{equation}
Thus we can obtain that
\begin{equation}
\label{suo5}
\begin{array}{lll}
&m(\frac{1}{\tau\lambda}\|z^{k}-z^{k+1}\|^{2}+\lambda\|B(y^{k+1}-y^{k})\|^{2}+\|x^{k+1}-x^{k}\|_{\mathcal{S}}^{2}+\|y^{k+1}-y^{k}\|_{\mathcal{T}}^{2})\\[1.5mm]
\leq & m(\frac{1}{\tau\lambda}\|z^{k}-
z^{k+1}\|^{2}+\lambda\|B(y^{k+1}-y^{k})\|^{2})+\|x^{k+1}-x^{k}\|_{\mathcal{S}}^{2}+\|y^{k+1}-y^{k}\|_{\mathcal{T}}^{2}\\[1.5mm]
\leq &\frac{1}{\tau\lambda}(\|z^{k}-z^{*}\|^{2}-\|z^{k+1}-z^{*}\|^{2})+
\lambda(\|B(y^{k}-y^{*})\|^{2}-\|B(y^{k+1}-y^{*})\|^{2})\\[1.5mm]
\quad&+(\|x^{k}-x^{*}\|_{\mathcal{S}}^{2}-\|x^{k+1}-x^{*}\|_{\mathcal{S}}^{2})+
(\|y^{k}-y^{*}\|_{\mathcal{T}}^{2}-\|y^{k+1}-y^{*}\|_{\mathcal{T}}^{2}).
\end{array}
\end{equation}
Summing \eqref{suo5} from $k=1$ to $K$, and making use of \eqref{dandiao}, one has 
\begin{equation}
\label{fangsuo1}
\begin{array}{lll}
& K(\frac{1}{\tau\lambda}\|z^{K}-z^{K+1}\|^{2}+\lambda\|B(y^{K+1}
-y^{K})\|^{2}+\|x^{K+1}-x^{K}\|_{\mathcal{S}}^{2}+
\|y^{K+1}-y^{K}\|_{\mathcal{T}}^{2})\\[1.5mm]
\leq &\sum_{k=1}^{K}(\frac{1}{\tau\lambda}\|z^{k}-z^{k+1}\|^{2}+
\lambda\|B(y^{k+1}-y^{k})\|^{2}+\|x^{k+1}-x^{k}\|_{\mathcal{S}}^{2}+
\|y^{k+1}-y^{k}\|_{\mathcal{T}}^{2})\\[1.5mm]
\leq &C_{2}:=\frac{1}{m}(\frac{1}{\tau\lambda}\|z^{1}-z^{*}\|^{2}+\lambda\|B(y^{1}-y^{*})\|^{2}+\|x^{1}-x^{*}\|_{\mathcal{S}}^{2}+\|y^{1}-y^{*}\|_{\mathcal{T}}^{2}).
\end{array}
\end{equation}
\end{itemize}
Finally, we discuss $\tau\in (0,1]$ uniformly.
Let $C:=\max\{C_{1},C_{2}\}$. Then, on the one hand, using a simple reduction of \eqref{suo11} and \eqref{fangsuo1}, for $K=1,2,\cdots$, one has 
\begin{equation}
\label{fangsuo2}
\begin{array}{rll}
\tau\lambda \|Ax^{K+1}+By^{K+1}-c\|&=\|z^{K+1}-z^{K}\|\leq \sqrt{\frac{\tau\lambda C}{K}},\\[1.5mm]
\|B(y^{K+1}-y^{K})\|\leq \sqrt{\frac{C}{\lambda K}},\quad \|y^{K+1}-&y^{K}\|_{\mathcal{T}}\leq \sqrt{\frac{C}{K}},~~\text{and}~~\|x^{K+1}-x^{K}\|_{\mathcal{S}}\leq \sqrt{\frac{C}{K}}.
\end{array}
\end{equation}
On the other hand, using \eqref{yuan1} and \eqref{suo3} we have
\begin{equation*}
\begin{array}{rll}
&\frac{1}{\tau\lambda}\|z^{k+1}-z^{*}\|^{2}+\lambda\|B(y^{k+1}-y^{*})\|^{2}+
\|x^{k+1}-x^{*}\|_{\mathcal{S}}^{2}+
\|y^{k+1}-y^{*}\|_{\mathcal{T}}^{2}\\[1.5mm]
\leq &\frac{1}{\tau\lambda}\|z^{k}-z^{*}\|^{2}+\lambda\|B(y^{k}-y^{*})\|^{2}+
\|x^{k}-x^{*}\|_{\mathcal{S}}^{2}+
\|y^{k}-y^{*}\|_{\mathcal{T}}^{2}\\[1.5mm]
\leq & \frac{1}{\tau\lambda}\|z^{1}-z^{*}\|^{2}+\lambda\|B(y^{1}-y^{*})\|^{2}+
\|x^{1}-x^{*}\|_{\mathcal{S}}^{2}+
\|y^{1}-y^{*}\|_{\mathcal{T}}^{2}\leq C.
\end{array}
\end{equation*}
Thus we obtain
\begin{equation}
\label{fangsuo3}
\begin{array}{rll}
\|z^{k+1}-z^{*}\|\leq \sqrt{\tau\lambda C},&\quad \|B(y^{k+1}-y^{*})\|\leq \sqrt{\frac{C}{\lambda}},\\[1.5mm]
\|x^{k}-x^{*}\|_{\mathcal{S}}\leq \sqrt{C},&~~\text{and}~~ \|y^{k}-y^{*}\|_{\mathcal{T}}\leq \sqrt{C}.
\end{array}
\end{equation}
By utilizing the contractions in \eqref{tutu}, \eqref{fangsuo2} and \eqref{fangsuo3}, we can get
\begin{equation}
\label{result11}
\begin{array}{rll}
&~~~ f(x^{K+1})-f(x^{*})+g(y^{K+1})-g(y^{*})-\langle z^{*},Ax^{K+1}+By^{K+1}-c\rangle
\leq \frac{4C}{\sqrt{K}}+\frac{C}{K\sqrt{\tau}}.
\end{array}
\end{equation}
Moreover, using \eqref{result11} and the fact that $ \|Ax^{K+1}+By^{K+1}-c\| \leq \sqrt{\frac{C}{\tau\lambda K}},$ one can see that for all $K=1,2\cdots$, it holds htat
\begin{equation*}
\begin{array}{rll}
-\|z^{*}\|\sqrt{\frac{C}{\tau\lambda K}}\leq f(x^{K+1})-f(x^{*})+g(y^{K+1})-g(y^{*})\leq \|z^{*}\|\sqrt{\frac{C}{\tau\lambda K}}+\frac{4C}{\sqrt{K}}+\frac{C}{K\sqrt{\tau}}.
\end{array}
\end{equation*}
This completes the proof. 
\qed


\end{proof}

\end{appendix}

\end{document}